\documentclass[reqno,b5paper]{amsart}

\usepackage{amsmath}
\usepackage{amssymb}
\usepackage{amsthm}
\usepackage{enumerate}
\usepackage[mathscr]{eucal}
\usepackage{eqlist}

\theoremstyle{plain}
\newtheorem{theorem}{Theorem}[section]
\newtheorem{prop}[theorem]{Proposition}
\newtheorem{lemma}{Lemma}[section]
\newtheorem{corol}{Corollary}[theorem]

\theoremstyle{definition}
\newtheorem{definition}{Definition}[section]

\theoremstyle{remark}
\newtheorem{example}{Example}

\numberwithin{equation}{section}

\begin{document}
\title[$n$-Weak Module Amenability]%
{$n$-Weak Module Amenability of Triangular Banach Algebras}
\author[Abasalt Bodaghi \and Ali Jabbari]%
{Abasalt Bodaghi* \and Ali Jabbari**}

\newcommand{\acr}{\newline\indent}

\address{\llap{*\,}Department of Mathematics\acr
                   Garmsar Branch\acr
                   Islamic Azad University\acr
                   Garmsar, IRAN}
\email{abasalt.bodaghi@gmail.com}

\address{\llap{**\,}Young Researchers Club\acr
                    Ardabil Branch\acr
                    Islamic Azad University\acr
                    Ardabil, IRAN}
\email{jabbari\_al@yahoo.com}

\thanks{}

\subjclass[2010]{Primary 46H20; Secondary 46H25, 16E40.}
\keywords{Inverse semigroups, module amenability, $n$-weak module
amenability, triangular Banach algebra}

\begin{abstract}
Let $\mathcal A$,
$\mathcal B$ be Banach $\mathfrak A$-modules with compatible
actions and $\mathcal M$ be a left Banach $\mathcal A$-$\mathfrak
A$-module and a right Banach $\mathcal B$-$\mathfrak A$-module. In
the current paper, we study module amenability, $n$-weak module
amenability and module Arens regularity of the triangular Banach
algebra $\mathcal T=\left[\begin{array}{cc}
                                           \mathcal A & \mathcal M \\
                                            &  \mathcal B\\
                                         \end{array}
                                       \right]$ (as an $\mathfrak T:=\Big\{\left[
                                         \begin{array}{cc}
                                           \alpha &  \\
                                            &  \alpha\\
                                         \end{array}
                                       \right]~|~\alpha\in\mathfrak
                                       A\Big\}$-module).
We employ these results to prove that for an inverse semigroup $S$
with subsemigroup $E$ of idempotents, the triangular Banach
algebra $\mathcal T_0=\left[\begin{array}{cc}
                                           \ell^1(S)& \ell^1(S) \\
                                            &  \ell^1(S)\\
                                         \end{array}
                                       \right]$ is permanently weakly module amenable (as an $\mathfrak T_0=\left[\begin{array}{cc}
                                           \ell^1(E)&  \\
                                            &  \ell^1(E)\\
                                         \end{array}
                                       \right]$-module). As an
example, we show that $\mathcal T_0$ is $\mathfrak T_0$-module
Arens regular if and only if the maximal group homomorphic image
$G_S$ of $S$ is finite.\end{abstract}

\maketitle

\section{Introduction}
The concept of module amenability for Banach algebras was initiated by Amini in
\cite{am1}.  The fundamental result was that the semigroup
algebra $\ell^1(S)$ is module amenable as a Banach module on
$\ell^1(E)$ if and only if $S$ is amenable, where $S$ is an
inverse semigroup with subsemigroup $E$ of idempotents. In fact
he showed that Johnson's theorem \cite{joh} (for groups) holds for discrete inverse semigroups if the relevant module structure is taken into account.
Amini and Bagha in \cite{am3} introduced the concept of weak
amenability for Banach algebras showed that $\ell^{1}(S)$ is
weakly $\ell ^{1}(E)$-module amenable when $S$ is a commutative
inverse semigroup with the set of idempotents $E$ (indeed this is
true for inverse semigroups whose idempotents are central).
Bodaghi et al. in \cite{bod} and \cite{bab} extended this result
and showed that $\ell ^{1}(S)$ is $n$-weakly module amenable as
an $\ell ^{1}(E)$-module (with trivial left action) when $n$ is
odd.

Pourabbas and Nasrabadi investigated weak module amenability of a
class of Banach algebras, called triangular Banach
algebras in \cite{po}. They considered the case where $\mathcal
A$, $\mathcal B$ are unital Banach algebras (with $\mathfrak
A$-module structure) and $\mathcal M$ is a unital Banach $\mathcal
A,\mathcal B$-module and showed that the corresponding triangular
Banach algebra $\mathcal T=\left[\begin{array}{cc}
                                           \mathcal A & \mathcal M \\
                                            &  \mathcal B\\
                                         \end{array}
                                       \right]$ is weakly module
                                       amenable (as an $\mathfrak T:=\Big\{\left[
                                         \begin{array}{cc}
                                           \alpha &  \\
                                            &  \alpha\\
                                         \end{array}
                                       \right]~|~\alpha\in\mathfrak
                                       A\Big\}$-module) if and only if
                                       both $\mathcal A$ and
$\mathcal B$ are weakly module amenable (as $\mathfrak
A$-modules). This can be regarded as the module version of
a result of Forrest and Marcoux \cite[Corollary 3.5]{fo1} (the
case that $\mathcal A$ or $\mathcal B$ has a bounded approximate
identity and $\mathcal M$ is essential was later proved by
Medghalchi et al. in \cite{me}). Also, they generalized  this
result to the case of $(2n-1)$-weak module amenability for $n\geq 1$
in \cite[theorem 3.7]{fo1}.

The concept of Arens module regularity is introduced in \cite{re}
and modified in \cite{pou}. It is shown that $\mathcal
A=\ell^1(S)$ is module Arens regular (as an $\ell ^{1}(E)$-module)
if and only if the group homomorphic image $G_S$ of $S$
is finite (see also \cite{pou2}).

The motivation of writing this paper is
{\it Example} \ref{123} which shows that for the commutative inverse
semigroup $S$ with subsemigroup $E$ of idempotents, the Banach algebra
$\mathcal T_0=\left[\begin{array}{cc}
                                           \ell ^{1}(S) & \ell ^{1}(S) \\
                                            &  \ell ^{1}(S)\\
                                         \end{array}
                                       \right]$ is $n$-weakly module amenable (as an $\mathfrak T_0:=\Big\{\left[
                                         \begin{array}{cc}
                                          \alpha &  \\
                                            &  \alpha\\
                                         \end{array}
                                       \right]~|~\alpha\in\ell
^{1}(E)\Big\}$-module) when $n\in \mathbb N$.

The paper is organized as follows: Section 2 is devoted to the study of module amenability of triangular Banach algebras.  The main result of section 3 asserts that module Arens
regularity of $\mathcal T=\left[\begin{array}{cc}
                                           \mathcal A & \mathcal M \\
                                            &  \mathcal B\\
                                         \end{array}
                                       \right]$ is  equivalent to that $\mathcal A$ and $\mathcal B$ are both module Arens regular and both act module regularly on $\mathcal M$. In section 4, we generalize some results of \cite{fo1} and \cite{me}, and we show that the triangular Banach algebra $\mathcal T$ is
$(2n-1)$-weakly module amenable (as an $\mathfrak T $-bimodule) if
and only if $\mathcal A$ and $\mathcal B$ are $(2n-1)$-weakly
module amenable (as Banach $\mathfrak A $-bimodules). In section
5, we prove that if $\mathcal A$ and $\mathcal B$ are
$(2n)$-weakly module amenable, then the first module cohomology
group of ${\mathcal T}$ with coefficients in $\mathcal T^{(2n)}$
is a quotient of a special set of module homomorphism from
$\mathcal M$ to $\mathcal M^{(2n)}$. In section 6, we show that
for a commutative inverse semigroup $S$ with the set of
idempotents $E$, the semigroup algebra $\ell^1(S)$ is $n$-weakly
module amenable as an $\ell^1(E)$-module for all
$n\in\mathbb{N}$. As a corollary, we show that $\mathcal
T_0=\left[\begin{array}{cc}
                                           \ell ^{1}(S) & \ell ^{1}(S) \\
                                            &  \ell ^{1}(S)\\
                                         \end{array}
                                       \right]$
is permanently weakly module amenable (as an $\mathfrak
T_0=\left[\begin{array}{cc}
                                           \ell ^{1}(E) &  \\
                                            &  \ell ^{1}(E)\\
                                         \end{array}
                                       \right]$-module). Finally,
we show that$\mathcal T_0$ is  module Arens regular if and only if
$G_S=S/\approx$ is finite, where $s\approx t$ whenever
$\delta_s-\delta_t$ belongs to the closed linear span of the set
$\{\delta_{set}-\delta_{st}: s,t\in S, e\in E\}.$

\section{Module amenability}
Let ${\mathcal A}$ and ${\mathfrak A}$ be Banach algebras such
that ${\mathcal A}$ is a Banach ${\mathfrak A}$-bimodule with
compatible actions, that is

\begin{equation}\label{1e1} \alpha\cdot(ab)=(\alpha\cdot a)b,
\,\,(ab)\cdot\alpha=a(b\cdot\alpha) \end{equation}  for all $a,b
\in {\mathcal A},\alpha\in {\mathfrak A}$. Let ${X}$ be a Banach
${\mathcal A}$-bimodule and a
 Banach ${\mathfrak A}$-bimodule with compatible actions, that
 is
 \begin{equation}\label{1e2} \alpha\cdot(a\cdot x)=(\alpha\cdot a)\cdot x,\,
(\alpha\cdot x)\cdot a=\alpha\cdot(x\cdot a) \end{equation}
\begin{equation}\label{1e3} x\cdot(a\cdot \alpha)=(x\cdot a)\cdot \alpha,\,
(a\cdot x)\cdot \alpha=a\cdot(x\cdot \alpha)\end{equation}
\begin{equation}\label{1e4}a\cdot(\alpha\cdot x)=(a\cdot\alpha)\cdot x,\,
x\cdot(\alpha\cdot a)=(x\cdot\alpha)\cdot a \end{equation} for all
$a \in{\mathcal A},\alpha\in {\mathfrak A},x\in{X}$. Then we say
that ${X}$ is a Banach ${\mathcal A}$-${\mathfrak A}$-module. If
$$\alpha\cdot x=x\cdot\alpha \hspace{0.3cm}( \alpha\in {\mathfrak
A},x\in{ X} )$$ then $ X $ is called a {\it commutative}
${\mathcal A}$-${\mathfrak A}$-module. Moreover, if
$$a.x = x.a \qquad (a\in \mathcal A, x\in X),$$
then $X$ is called a {\it bi-commutative} Banach $\mathcal A$-$\mathfrak A$-module.

If $X $ is a commutative
Banach ${\mathcal A}$-${\mathfrak A}$-module, then so is $X^*$,
where the actions of $\mathcal A$ and ${\mathfrak A}$ on $ X^*$
are defined as usual:
$$\langle
f\cdot\alpha,x\rangle{}=\langle{}f,\alpha\cdot
x\rangle{},\,\,\langle{} f\cdot a,x\rangle{}=\langle{}f,a\cdot
x\rangle,$$
$$\langle\alpha\cdot f,x\rangle{}=\langle{}f,x\cdot\alpha\rangle{},\,\,\langle{}
a\cdot f,x\rangle{}=\langle{}f,x\cdot a\rangle{}\quad (a
\in{\mathcal A},\alpha\in {\mathfrak A},x\in{ X},f \in X^* ).$$

 Note that when ${\mathcal
A}$ acts on itself by algebra multiplication, it is not in
general a Banach ${\mathcal A}$-${\mathfrak A}$-module, as we
have not assumed the compatibility condition
$$a\cdot(\alpha\cdot b)=(a\cdot\alpha)\cdot b\quad (\alpha\in {\mathfrak A}, a,b \in{\mathcal
A}).$$

If $\mathcal A$ is a commutative $\mathfrak A$-module and acts on
itself by multiplication from both sides, then it is also a
Banach ${\mathcal A}$-${\mathfrak A}$-module. Also, if $\mathcal A$ is a
commutative Banach algebra, then it is a bi-commutative ${\mathcal A}$-${\mathfrak A}$-module. In
these cases, $\mathcal A^{(n)}$, $n$th dual space of $\mathcal A$ is also a commutative or bi-commutative ${\mathcal A}$-${\mathfrak A}$-module, respectively.

Let ${\mathcal A}$ and  ${\mathcal B}$ be Banach ${\mathfrak
A}$-modules with compatible actions (\ref{1e1}). Then a left
${\mathfrak A}$-module map is a mapping $T:\mathcal
A\longrightarrow \mathcal B$ with $ T(a \pm b)=T(a)\pm T(b)$ and
 $T (\alpha \cdot x)= \alpha \cdot T(x)$
for $a \in {\mathcal A}, b \in {\mathcal B},$ and $\alpha \in
{\mathfrak A}$.  A right or two-sided ${\mathfrak A}$-module map
is defined similarly.

Let ${\mathcal A}$ and ${\mathfrak A}$ be as above and $X$ be a
Banach ${\mathcal A}$-$\mathfrak A$-module. A ($\mathfrak A$-){\it module
derivation} is a bounded $\mathfrak A$-module map $D: \mathcal A
\longrightarrow X $ satisfying
$$D(ab)=D(a)\cdot b+a\cdot D(b) \qquad
(a,b \in \mathcal A).$$

One should note that $D$ is not necessarily linear, but its
boundedness (defined as the existence of $M>0$ such that
$\|D(a)\|\leq M\|a\|$, for all $a\in \mathcal A$) still implies
its continuity, as it preserves subtraction. When $X $ is
commutative, each $x \in X $ defines a module derivation
$$D_x(a)=a\cdot x-x\cdot a \qquad (a \in{\mathcal A}).$$

These are called {\it inner} ${\mathfrak A}$-module derivations.

We use notations $\mathcal Z_{\mathfrak A}^1(\mathcal A, X)$
and $\mathcal N_{\mathfrak A}^1(\mathcal A, X)$ for the set of all
module derivations and inner derivations from $\mathcal A$ to $X$,
respectively. Also the quotient space $\mathcal Z_{\mathfrak
A}^1(\mathcal A, X)/\mathcal N_{\mathfrak A}^1(\mathcal A, X)$ (which
we call the first  ${\mathfrak A}$-module cohomology group of
${\mathcal A}$ with coefficients in $X$) is denoted by $\mathcal
H^1_{\mathfrak A}(\mathcal A, X)$. A Banach algebra ${\mathcal A}$
is module amenable if $\mathcal H^1_{\mathfrak A}(\mathcal A,
X^{*})=\{0\}$, for each commutative Banach ${\mathcal
A}$-${\mathfrak A}$-module $X$ \cite{am1}.

It is proved in \cite[Proposition 2.5]{am1} that the homomorphic
image of a module amenable Banach algebra under a continuous
homomorphism with dense range is also a module amenable Banach
algebra. This fact leads to the following result.

\begin{lemma}\label{corm} Let $\mathcal A$ be a Banach
${\mathfrak A}$-module and $\mathcal I$ be a closed ideal in
$\mathcal A$. Then module amenability of $\mathcal A$ implies
module amenability of $\mathcal A/\mathcal I$.\end{lemma}

\begin{prop} \label{ropo}
Let $\mathcal A$ be a Banach ${\mathfrak A}$-module, $\mathcal I$
be a closed ideal and ${\mathfrak A}$-submodule of $\mathcal A$.
If $\mathcal I$ and  $\mathcal A/\mathcal I$ are  module
amenable, then so is $\mathcal A$.
\end{prop}
\begin{proof} Assume that $X$ be a commutative Banach $\mathcal A$-$\mathfrak A$-module with
compatible actions and $D:\mathcal A\longrightarrow X^*$ be a
bounded module derivation. Since $\mathcal I$ is module amenable,
there exists $f_1\in X^*$ such that $D\mid_{\mathcal I}=D_{f_1}$.
Thus, the map $\widetilde{D}=D-D_{f_1}$ vanishes on $\mathcal I$.
This map induces a module derivation from $\mathcal A/\mathcal I$ into $X^*$, which we again
denote by $\widetilde{D}$. Let $Y$ be the closed linear span of
$$\{a\cdot x- y\cdot b \mid a,b \in \mathcal I, x,y\in X \},$$
in $X$. It follows immediately that $Y$ is a closed ${\mathcal
A}$-submodule and ${\mathfrak A}$-submodule of $X$, and so $X/Y$
is a Banach $\mathcal A/\mathcal I$-$\mathfrak A$-module with
compatible actions. Since $D\mid_{\mathcal I}=\{0\}$, we have
$a\cdot\widetilde{D}(b)=\widetilde{D}(ab)-\widetilde{D}(a)\cdot
b=0$ for all  $a \in \mathcal I$ and $b \in \mathcal A$.
Similarly, $ \widetilde{D}(b)\cdot a=0$. This implies
$\widetilde{D}(\mathcal A/\mathcal I)\subset Y^{\perp}=(X/Y)^*$.
Due to module amenability of $\mathcal A/\mathcal I$, there is
$f_2\in Y^{\perp}\subset X^*$ such that $\widetilde{D}=D_{f_2}$.
Consequently, $D=D_{f_1+f_2}$.
\end{proof}

\begin{prop} \label{th2}
Let $\mathcal A$  and $\mathcal B$ be Banach $\mathfrak
A$-modules. Then $\mathcal A\oplus_{\ell^1}\mathcal B$,
$\ell^1$-direct sum of $\mathcal A$  and $\mathcal B$ is module
amenable if and only if $\mathcal A$ and $\mathcal B$ are module
amenable.
\end{prop}
\begin{proof}
Let $\mathcal A$ and $\mathcal B$ be module amenable. Since
$\mathcal B$, the closed ideal of $\mathcal
A\oplus_{\ell^1}\mathcal B$ and the quotient algebra $(\mathcal
A\oplus_{\ell^1}\mathcal B)/\mathcal B\cong\mathcal A$ are module
amenable, $\mathcal A\oplus_{\ell^1}\mathcal B$ is module
amenable by Proposition \ref{ropo}.

Conversely, assume that $\mathcal A\oplus_{\ell^1}\mathcal B$ is
module amenable. Thus $(\mathcal A\oplus_{\ell^1}\mathcal
B)/\mathcal A\cong\mathcal B$ and $(\mathcal
A\oplus_{\ell^1}\mathcal B)/\mathcal B\cong\mathcal A$ are module
amenable by Lemma \ref{corm}.
\end{proof}

Let ${\mathcal A}$ and ${\mathfrak A}$ be Banach ${\mathfrak
A}$-bimodules with compatible actions (\ref{1e1}) and $\mathcal M$ be a
Banach ${\mathcal A}$-bimodule and a
 Banach ${\mathfrak A}$-bimodule with compatible actions (\ref{1e2}) and (\ref{1e3}).

Let $J_\mathcal M$ be a subspace of $\mathcal M$  generated by elements of the
form $a\cdot(\alpha\cdot x)-(a\cdot\alpha)\cdot x$ and
$x\cdot(\alpha\cdot a)-(x\cdot\alpha)\cdot a$ for $ \alpha\in
{\mathfrak A},a\in{\mathcal A}$ and $x\in \mathcal M$. One can see from
(\ref{1e2}) and (\ref{1e3}) that $J_\mathcal M$ is an ${\mathcal
A}$-submodule and ${\mathfrak A}$-submodule of $\mathcal M$. Note that
the equalities (\ref{1e4}) do not necessarily hold for $\mathcal M$. When $\mathcal M=\mathcal
A$, $J_{\mathcal A}$ is in fact the closed ideal of ${\mathcal A}$
generated by elements of the form $ (a\cdot\alpha) b-a(\alpha
\cdot b)$ for $ \alpha\in {\mathfrak A},a,b\in{\mathcal A}$ (see
\cite{bod}, \cite{bab} and \cite{pou}).

Let $\mathcal T= \Big\{\left[
                                         \begin{array}{cc}
                                           a &  m\\
                                            &  b\\
                                         \end{array}
                                       \right]~|~a\in\mathcal A, b\in \mathcal B, m\in \mathcal M\Big\}
$ be triangular Banach algebra equipped with the usual 2 $\times$
2 matrix addition and formal multiplication. The norm on
$\mathcal T$ is $\|\left[
                                         \begin{array}{cc}
                                           a &  m\\
                                            &  b\\
                                         \end{array}
                                       \right]\|=\|a\|_{\mathcal A}+\|b\|_{\mathcal B}+\|m\|_{\mathcal M}$. Now, let $\mathcal A$, $\mathcal B$ and $\mathcal M$  be  Banach
$\mathfrak A$-bimodules, and let $\mathcal M$ be a Banach
$\mathcal A,\mathcal B$-module (left $\mathcal A$-module and
right $\mathcal B$-module). Similar to \cite{po}, we consider the
Banach algebra $\mathfrak T:=\Big\{\left[
                                         \begin{array}{cc}
                                           \alpha &  \\
                                            &  \alpha\\
                                         \end{array}
                                       \right]~|~\alpha\in\mathfrak A\Big\}
$. Then, the Banach algebra $\mathcal T=\left[\begin{array}{cc}
                                           \mathcal A & \mathcal M \\
                                            &  \mathcal B\\
                                         \end{array}
                                       \right]$
with usual $2\times2$ matrix product is a $\mathfrak T$-bimodule.
In fact,  That is isomorphic to $\mathcal A\oplus_{\ell^1}\mathcal
M\oplus_{\ell^1}\mathcal B$ as a Banach space and a Banach
$\mathfrak A$-bimodule. The following result is a module version of \cite[Theorem 4.2]{me} and the
proof is similar. However, we bring its proof.

\begin{theorem} \label{te}
Let $\mathcal A,\mathcal B$ be Banach $\mathfrak A$-modules and
$\mathcal M$ be commutative Banach $\mathfrak A$-module. The
triangular Banach algebra $\mathcal T$ is module amenable
\emph{(}as an $\mathfrak T$-bimodule\emph{)} if and only if
$\mathcal A$ and $\mathcal B$ are module amenable \emph{(}as
Banach $\mathfrak A$-bimodules\emph{)} and $\mathcal M=0$.
\end{theorem}
\begin{proof} First note that when $\mathcal T$ is $\mathfrak
T$-bimodule it means that $\mathcal A\oplus_{\ell^1}\mathcal
M\oplus_{\ell^1}\mathcal B$ is $\mathfrak A$-bimodule with the
usual actions. Now, let $\mathcal A$ and $\mathcal B$ are module
amenable and $\mathcal M=0$. Then $\mathcal
T=\left[\begin{array}{cc}
                                           \mathcal A & 0 \\
                                            &  \mathcal B\\
                                         \end{array}
                                       \right]$ is module amenable
                                       by Proposition \ref{th2}.

Conversely, assume that $\mathcal T$ is $\mathfrak T$-module
amenable. The Banach algebras $\left[\begin{array}{cc}
                                           \mathcal A & \mathcal M \\
                                            &  0\\
                                         \end{array}
                                       \right]$ and $\left[\begin{array}{cc}
                                           0 & \mathcal M  \\
                                            &  \mathcal B\\
                                         \end{array}
                                       \right]$ are closed ideals
                                       of $\mathcal T$, and thus $\mathcal
                                       A$ and $\mathcal B$ are $\mathfrak A$-module
amenable by Lemma \ref{corm}. Since $\mathcal M$ is
complemented in $\mathcal T$, it is module amenable and since
$\mathcal M$ is a commutative $\mathfrak A$-module, it has a
bounded approximate identity by \cite[Proposition 2.2]{am1}, hence
it should be zero.
\end{proof}


\section{Module Arens regularity}

For given $\Gamma_1, \Gamma_2\in\mathcal A^{**}$, by the Goldstein
theorem there are nets $(a_{1,i})_i$ and $(a_{2,j})_j$ in
$\mathcal A$ such that $\Gamma_1=w^*-\lim_i a_{1,i}$ and
$\Gamma_2=w^*-\lim_j a_{2,j}$. Then we consider the first Arens
product on $\mathcal A^{**}$ as follows:
 \begin{equation*}
    \Gamma_1\Box\Gamma_2=w^*-\lim_i\lim_j a_{1,i}a_{2,j},
 \end{equation*}
 and similarly for any $\Psi_1,\Psi_2\in\mathcal B^{**}$, there exist nets $(b_{1,i})_i$ and $(b_{2,j})_j$ such that the first Arens product on $\mathcal B$
  is defined as follows:
 \begin{equation*}
    \Psi_1\Box\Psi_2=w^*-\lim_i\lim_j b_{1,i}b_{2,j}.
 \end{equation*}

We extend the actions of $\mathcal A$ and $\mathcal B$ on $\mathcal M$ to actions of $\mathcal A^{**}$ and
$\mathcal B^{**}$ on $\mathcal M^{**}$ via
  \begin{equation*}
    \Gamma\Box\Pi=w^*-\lim_i\lim_k a_{i}\cdot x_{k},\hspace{1cm}\emph{\emph{and}}\hspace{1cm} \Pi\Box\Psi=w^*-\lim_k\lim_j x_{k}\cdot b_{j},
 \end{equation*}
 where $\Gamma=w^*-\lim_i a_{i}$, $\Psi=w^*-\lim_j b_{j}$, and $\Pi=w^*-\lim_k x_{k}$. We define the first Arens product on $\mathcal T^{**}$ in a natural way. Let $T_1=\left[\begin{array}{cc}
                                           \Gamma_1 & \Pi_1 \\
                                            &  \Psi_1\\
                                         \end{array}
                                       \right], T_2=\left[\begin{array}{cc}
                                           \Gamma_2 & \Pi_2 \\
                                            &  \Psi_2\\
                                         \end{array}
                                       \right]\in\mathcal T^{**}$ so that $T_1=w^*-\lim_i\left[\begin{array}{cc}
                                           a_{1,i} & x_{1,i} \\
                                            &  b_{1,i}\\
                                         \end{array}
                                       \right]$ and $T_2=w^*-\lim_j\left[\begin{array}{cc}
                                           a_{2,j} & x_{2,j} \\
                                            &  b_{2,j}\\
                                         \end{array}
                                       \right]$.  Then we have
                                       \begin{eqnarray*}
                                          T_1\Box T_2&=&  \left[\begin{array}{cc}
                                           \Gamma_1 & \Pi_1 \\
                                            &  \Psi_1\\
                                         \end{array}
                                       \right]\Box \left[\begin{array}{cc}
                                           \Gamma_2 & \Pi_2 \\
                                            &  \Psi_2\\
                                         \end{array}
                                       \right]\\
                                       &=& w^*-\lim_i\lim_j \left[\begin{array}{cc}
                                           a_{1,i} & x_{1,i} \\
                                            &  b_{1,i}\\
                                         \end{array}
                                       \right] \left[\begin{array}{cc}
                                           a_{2,j} & x_{2,j} \\
                                            &  b_{2,j}\\
                                         \end{array}
                                       \right]\\
                                          &=& w^*-\lim_i\lim_j \left[\begin{array}{cc}
                                           a_{1,i}a_{2,j} & a_{1,i}x_{2,j}+x_{1,i}b_{2,j} \\
                                            &  b_{1,i}b_{2,j}\\
                                         \end{array}
                                       \right] \\
                                       &=& \left[\begin{array}{cc}
                                           \Gamma_1\Box\Gamma_2 & \Gamma_1\Box\Pi_2+\Pi_1\Box\Psi_2 \\
                                            &  \Psi_1\Box\Psi_2\\
                                         \end{array}
                                       \right].
                                       \end{eqnarray*}

Similarly, we consider the second Arens product on $\mathcal
A^{**}$, $\mathcal B^{**}$ and module actions $\mathcal A^{**}$,
$\mathcal B^{**}$ on $\mathcal M^{**}$ as follows:
 \begin{equation*}
    \Gamma_1\diamond\Gamma_2=w^*-\lim_j\lim_i a_{1,i}a_{2,j},\hspace{0.5cm}\Psi_1\diamond\Psi_2=w^*-\lim_j\lim_i b_{1,i}b_{2,j},
 \end{equation*}
and
\begin{equation*}
    \Gamma_1\diamond\Pi=w^*-\lim_k\lim_i a_{1,i}x_{k},\hspace{0.5cm} \Pi\diamond\Psi_1=w^*-\lim_j\lim_k x_{k}b_{1,j},
 \end{equation*}
 where $\Gamma_1,\Gamma_2\in\mathcal A^{**}$, $\Psi_1,\Psi_2\in\mathcal B^{**}$ and $\Pi\in \mathcal M^{**}$. Thus the second Arens product on $\mathcal T^{**}$ can be defined via
 \begin{eqnarray*}
                                          T_1\diamond T_2&=&  \left[\begin{array}{cc}
                                           \Gamma_1 & \Pi_1 \\
                                            &  \Psi_1\\
                                         \end{array}
                                       \right]\diamond \left[\begin{array}{cc}
                                           \Gamma_2 & \Pi_2 \\
                                            &  \Psi_2\\
                                         \end{array}
                                       \right]\\
                                       &=& w^*-\lim_j\lim_i \left[\begin{array}{cc}
                                           a_{1,i} & x_{1,i} \\
                                            &  b_{1,i}\\
                                         \end{array}
                                       \right] \left[\begin{array}{cc}
                                           a_{2,j} & x_{2,j} \\
                                            &  b_{2,j}\\
                                         \end{array}
                                       \right]\\
                                          &=& w^*-\lim_j\lim_i \left[\begin{array}{cc}
                                           a_{1,i}a_{2,j} & a_{1,i}x_{2,j}+x_{1,i}b_{2,j} \\
                                            &  b_{1,i}b_{2,j}\\
                                         \end{array}
                                       \right]\\
                                        &=& \left[\begin{array}{cc}
                                           \Gamma_1\diamond\Gamma_2 & \Gamma_1\diamond\Pi_2+\Pi_1\diamond\Psi_2 \\
                                            &  \Psi_1\diamond\Psi_2\\
                                         \end{array}
                                       \right].
                                       \end{eqnarray*}

The concept of module Arens regularity for Banach algebras is
defined in \cite{re}. A Banach algebra $\mathcal A$ is module
Arens regular as a Banach $\mathfrak A$-module if and only if
$\Gamma_1\Box\Gamma_2-\Gamma_1\diamond\Gamma_2\in J_\mathcal
A^{\bot\bot}$, for every $\Gamma_1,\Gamma_2\in\mathcal A^{**}$
(see \cite[Theorem 2.2]{re} and \cite[Theorem 3.3]{abeh}).

We say that the Banach algebras $\mathcal A$ and $\mathcal B$ act
{\it module regularly} on $\mathcal M$ if for each
$\Gamma\in\mathcal A^{**}$, $\Psi\in\mathcal B^{**}$ and
$\Pi\in\mathcal M^{**}$ we have
$$\Gamma\Box\Pi-\Gamma\diamond\Pi\in J_\mathcal M^{\bot\bot}, \quad \Pi\Box\Psi-\Pi\diamond\Psi\in J_\mathcal M^{\bot\bot}.$$
Recall that the Banach algebras $\mathcal A$ and $\mathcal B$ act
 regularly on $\mathcal A,\mathcal B$-module $\mathcal M$ if for every
$\Gamma\in\mathcal A^{**}$, $\Psi\in\mathcal B^{**}$ and
$\Pi\in\mathcal M^{**}$ we have $\Gamma\Box\Pi=\Gamma\diamond\Pi$
and  $\Pi\Box\Psi=\Pi\diamond\Psi$ \cite{fo2}.  It follows from the above that $\mathcal A$ and
$\mathcal B$ act module regularly on $\mathcal M$ if and only if
$\mathcal A/J_\mathcal A$ and $\mathcal B/J_\mathcal B$ act
regularly on $\mathcal M/J_\mathcal M$. Indeed,
\begin{align*}
\Gamma\Box\Pi-\Gamma\diamond\Pi\in J_\mathcal M^{\bot\bot}
&\Longleftrightarrow \Gamma\Box\Pi+J_\mathcal M^{\bot\bot}=\Gamma\diamond\Pi+ J_\mathcal M^{\bot\bot}\\
&\Longleftrightarrow (\Gamma+J_\mathcal
A^{\bot\bot})\Box(\Pi+J_\mathcal M^{\bot\bot})=(\Gamma+J_\mathcal
A^{\bot\bot})\diamond(\Pi+J_\mathcal M^{\bot\bot}).
\end{align*}

Similarly, $\Pi\Box\Psi-\Pi\diamond\Psi\in J_\mathcal
M^{\bot\bot}$ if and only if $(\Pi+J_\mathcal
M^{\bot\bot})\Box(\Psi+J_\mathcal B^{\bot\bot})=(\Pi+J_\mathcal
M^{\bot\bot})\diamond(\Psi+J_\mathcal B^{\bot\bot})$. Similarly we can show that $J_\mathcal
T=\left[\begin{array}{cc}
                                           J_\mathcal A & J_\mathcal M \\
                                            &  J_\mathcal B\\
                                         \end{array}
                                       \right]$
and thus $\left[\begin{array}{cc}
                                           J_\mathcal A^{\perp} & J_\mathcal M^{\perp} \\
                                            &  J_\mathcal B^{\perp}\\
                                         \end{array}
                                       \right]\subseteq
                                       J_\mathcal T^{\perp}$.
It is shown in \cite[Lemma 3.1]{bab} that $J_\mathcal A$ and
$J_\mathcal B$ are the closed subspace of $\mathcal A$ an
$\mathcal B$ respectively. It is easy to check that $J_\mathcal M$
is also a closed subspace of $\mathcal M$. Hence $J_\mathcal A,
J_\mathcal B$ and $J_\mathcal M$ are weak$^*$-dense in
$J^{\perp\perp}_\mathcal A, J^{\perp\perp}_\mathcal B$ and
$J^{\perp\perp}_\mathcal M$ respectively by \cite[Theorem A.3.47]{da1}.
Hence $J_\mathcal A\oplus_{\ell^1}J_\mathcal
M\oplus_{\ell^1}J_\mathcal B$ is weak$^*$-dense in
$J^{\perp\perp}_\mathcal A\oplus_{\ell^1}J^{\perp\perp}_\mathcal
M\oplus_{\ell^1}J^{\perp\perp}_\mathcal B$. On the other hand,
$J_\mathcal A\oplus_{\ell^1}J_\mathcal M\oplus_{\ell^1}J_\mathcal
B$ is closed subspace of $\mathcal A\oplus_{\ell^1}\mathcal
M\oplus_{\ell^1}\mathcal B$ and so it is weak$^*$-dense in
$J^{\perp\perp}_\mathcal T$. Therefore $J_\mathcal T^{\perp\perp}=
\left[\begin{array}{cc}
                                           J_\mathcal A^{\perp\perp} & J_\mathcal M^{\perp\perp} \\
                                            &  J_\mathcal B^{\perp\perp}\\
                                         \end{array}
                                       \right]$. Summing up:
\begin{theorem} \label{the2}
Let $\mathcal A$  and $\mathcal B$ be  Banach $\mathfrak
A$-modules. The triangular Banach algebra $\mathcal T$ is
 module Arens regular \emph{(}as an
$\mathfrak T$-bimodule\emph{)} if and only if $\mathcal A$ and
$\mathcal B$ are module Arens regular \emph{(}as Banach
$\mathfrak A $-bimodules\emph{)} and $\mathcal A$ and $\mathcal
B$ act module regularly on $\mathcal M$.
\end{theorem}


\section{(2n-1)-Weak Module Amenability}
We start this section with the definition of $n$-weak module
amenability which is introduced in \cite{bab}. When $\mathcal A$ is a commutative $\mathfrak A$-bimodule, we have
$J=\{0\}$ and $\mathcal A$ is a commutative $\mathcal A$-$\mathfrak A$-module. In this case, the
following definition (for $n = 1$) coincides with the definition of weak
module amenability in \cite{am3}. If $\mathcal A$ is a commutative Banach algebra
and a commutative $\mathfrak A$-bimodule with compatible actions, then $\mathcal A$ is a
bi-commutative $\mathcal A$-$\mathfrak A$-module. In this case, for each bi-commutative
Banach $\mathcal A$-$\mathfrak A$-module $X$, all bounded module derivation from $\mathcal A$ into
$X$ are zero (see the next lemma) and we get the definition of weak
module amenability for commutative Banach algebras as in \cite{am3}.

\begin{definition}\label{def1} Let ${\mathcal A}$ be a Banach algebra, $n\in \mathbb{N}$.
Then ${\mathcal A}$  is called $n$-{\it weakly module amenable}
(as an ${\mathfrak A}$-module) if $(\mathcal A/J)^{(n)}$ is a
commutative $\mathcal A$-$\mathfrak A$-module, and each module
derivation from $D:\mathcal A\longrightarrow (\mathcal A/J)^{(n)}$
is inner; that is, $D(a)=a\cdot y-y\cdot a=: D_y(a)$ for some
$y\in ({\mathcal A/J})^{(n)}$ and each $a\in A$. Also ${\mathcal
A}$ is called {\it weakly module amenable} if it is
 $1$-weakly module amenable and {\it permanently weakly module amenable} if
it is $n$-weakly module amenable for each $n\in \mathbb{N}$.
\end{definition}

\begin{lemma}\label{deff} 
Let $\mathcal A$ be a essential bi-commutative
$\mathcal A$-$\mathfrak A$-module. Then $\mathcal A$ is weakly module amenable (as an $\mathfrak A$-module) if and only if for
each bi-commutative Banach $\mathcal A$-$\mathfrak A$-module $X$, all bounded module derivation
from $\mathcal A$ into $X$ are zero.
\end{lemma}
\begin{proof} We follow the standard argument in \cite[Theorem 2.8.63]{da1}. Assume that there exists $D\in \mathcal Z_{\mathfrak A}(\mathcal A, X)$ with $D\neq 0$. Since $\overline{\mathcal A^2}=\mathcal A$, there exists $a_0\in \mathcal A$ such that $D(a_0^2)\neq 0$. We have $a_0\cdot D(a_0)\neq 0$ and thus $f\in X^*$ with $f(a_0\cdot D(a_0))=1$. Set $R:X\longrightarrow \mathcal A^*$ defined by $R(x)(a)=f(a\cdot x)$ where $a\in \mathcal A, x\in X$. It is easy to check that $R\circ D\in\mathcal Z_{\mathfrak A}(\mathcal A, \mathcal A^*)$. We get
$\langle R\circ D(a_0),a_0\rangle=\langle f, a_0\cdot D(a_0)\rangle=1$, and so $R\circ D\neq 0$. This shows that $\mathcal A$ is not weakly module amenable. The converse is clear.
\end{proof}

Let $\mathcal A$, $\mathcal B$ and $\mathcal M$  be as in the
previous section. If these are commutative Banach $\mathfrak
A$-modules, then the corresponding triangular Banach algebra
$\mathcal T=\left[\begin{array}{cc}
                                           \mathcal A & \mathcal M \\
                                            &  \mathcal B\\
                                         \end{array}
                                       \right]$
 is a commutative $\mathfrak T$-module in which $\mathfrak T:=\Big\{\left[
                                         \begin{array}{cc}
                                           \alpha &  \\
                                            &  \alpha\\
                                         \end{array}
                                       \right]~|~\alpha\in\mathfrak A\Big\}
$. Also $\mathcal T$ is isomorphic to $\mathcal
A\oplus_{\ell^1}\mathcal M\oplus_{\ell^1}\mathcal B$ as a Banach
$\mathfrak T$-module and a Banach $\mathfrak A$-module,
respectively. Therefore $\mathcal T^{(2n-1)}\simeq \mathcal
A^{(2n-1)}\oplus_{\ell^1}\mathcal
M^{(2n-1)}\oplus_{\ell^1}\mathcal B^{(2n-1)}$, while $\mathcal
T^{(2n)}\simeq \mathcal A^{(2n)}\oplus_{\ell^\infty}\mathcal
M^{(2n)}\oplus_{\ell^\infty}\mathcal B^{(2n)}$ in which $\mathcal
T^{(n)}$ is Banach $\mathfrak T$-module and $\mathcal
A^{(2n-1)}\oplus_{\ell^1}\mathcal
M^{(2n-1)}\oplus_{\ell^1}\mathcal B^{(2n-1)}$, $\mathcal
A^{(2n)}\oplus_{\ell^\infty}\mathcal
M^{(2n)}\oplus_{\ell^\infty}\mathcal B^{(2n)}$ are Banach
$\mathfrak A$-modules. Suppose that $\mathfrak t=
\left[\begin{array}{cc}
                                           a & m \\
                                            &  b\\
                                         \end{array}
                                       \right]\in \mathcal T$and $\theta = \left[\begin{array}{cc}
                                           f & \lambda \\
                                            &  g\\
                                         \end{array}
                                       \right]\in \mathcal T^*$. Then the pairing of $\mathcal T^*$ and $\mathcal T$ is given by
$\theta(\mathfrak t)=f(a)+\lambda(m)+g(b)$. Indeed, it is easy to
check that the module actions $\mathcal T$ on $\mathcal T^*$ are
as follows:
$$\mathfrak t\cdot\theta=\left[\begin{array}{cc}
                                           a\cdot f+m\cdot\lambda & b\cdot\lambda \\
                                            & b\cdot g\\
                                         \end{array}
                                       \right]\quad\text{and}\quad \theta\cdot\mathfrak t=\left[\begin{array}{cc}
                                           f\cdot a & \lambda\cdot a \\
                                            & \lambda\cdot m+g\cdot b\\
                                         \end{array}
                                       \right]. $$
We remove the dot for simplicity. This process may be repeated to define the actions of
$\mathcal T$ on $\mathcal T^{(n)}$ as follows:
\begin{equation*}\label{}
   \left[\begin{array}{cc}
                                           a & m \\
                                            &  b\\
                                         \end{array}
                                       \right]\cdot \left[\begin{array}{cc}
                                           \lambda & \gamma \\
                                            &  \mu\\
                                         \end{array}
                                       \right]=\left[\begin{array}{cc}
                                           a\lambda & a\gamma+m\mu \\
                                            &  b\mu\\
                                         \end{array}
                                       \right],\end{equation*}\begin{equation*}\left[\begin{array}{cc}
        \lambda & \gamma \\
         &  \mu\\
         \end{array}
          \right]\cdot\left[\begin{array}{cc}
                                           a & m \\
                                            &  b\\
                                         \end{array}
                                       \right]=\left[\begin{array}{cc}
                                          \lambda a & \lambda m +\gamma b\\
                                            & \mu b\\
                                         \end{array}
                                       \right]
\end{equation*}
and
\begin{equation*}
\left[\begin{array}{cc}
                                           a & m \\
                                            &  b\\
                                         \end{array}
                                       \right]\cdot \left[\begin{array}{cc}
                                          \phi & \varphi \\
                                            &  \psi\\
                                         \end{array}
                                       \right]=\left[\begin{array}{cc}
                                           a\phi+m\varphi& b\varphi \\
                                            &  b\psi\\
                                         \end{array}
                                       \right],\end{equation*}\begin{equation*}
    \left[\begin{array}{cc}
        \phi & \varphi \\
         &  \psi\\
         \end{array}
          \right]\cdot\left[\begin{array}{cc}
                                           a & m \\
                                            &  b\\
                                         \end{array}
                                       \right]=\left[\begin{array}{cc}
                                          \phi a & \varphi a\\
                                            & \psi b+\varphi m\\
                                         \end{array}
                                       \right],
\end{equation*}
for every $\left[\begin{array}{cc}
                                           a & m \\
                                            &  b\\
                                         \end{array}
                                       \right]\in\mathcal T, \left[\begin{array}{cc}
                                           \lambda & \gamma \\
                                            &  \mu\\
                                         \end{array}
                                       \right]\in\mathcal T^{(2n)}$ and $\left[\begin{array}{cc}
                                           \phi & \varphi \\
                                            &  \psi\\
                                         \end{array}
                                       \right]\in\mathcal T^{(2n-1)}$.

Henceforth, we assume that $\mathcal A$, $\mathcal B$ and
$\mathcal M$ are commutative Banach $\mathfrak A$-modules and thus
$\mathcal T^{(2n-1)}$ and $\mathcal T^{(2n)}$ become commutative
Banach $\mathfrak T$-modules for any $n\in \mathbb{N}$.

The following two lemmas are proved similar to Lemmas 1.1 and 1.2 in \cite{po}.
\begin{lemma}\label{po1}
The $\mathfrak T$-module map $D:\mathcal T\longrightarrow\mathcal
T^{(2n-1)}$ is a module derivation if and only if there exist
module derivations $D_\mathcal A:\mathcal
A\longrightarrow\mathcal A^{(2n-1)}$, $D_\mathcal B:\mathcal
B\longrightarrow\mathcal B^{(2n-1)}$ and $\gamma\in\mathcal
M^{(2n-1)}$ such that
                                       \begin{equation}\label{}
                                        D\Big{(}\left[\begin{array}{cc}
                                           a & m \\
                                            &  b\\
                                         \end{array}
                                       \right]\Big{)}=\left[\begin{array}{cc}
                                           D_\mathcal A(a)-m\gamma & \gamma a-b\gamma \\
                                            & D_\mathcal B(b)+\gamma m\\
                                         \end{array}
                                       \right],
                                       \end{equation}
for every $\left[\begin{array}{cc}
                                           a & m \\
                                            &  b\\
                                         \end{array}
                                       \right]\in\mathcal T$.
\end{lemma}
\begin{lemma}\label{po2}
Let $D_\mathcal A:\mathcal A\longrightarrow\mathcal A^{(2n-1)}$
and $D_\mathcal B:\mathcal B\longrightarrow\mathcal B^{(2n-1)}$ be
bounded module derivations. Then $D_{\mathcal A\mathcal
B}:\mathcal T\longrightarrow\mathcal T^{(2n-1)}$ defined via
\begin{equation}\label{}
 \left[\begin{array}{cc}
                                           a & m \\
                                            &  b\\
                                         \end{array}
                                       \right]\longmapsto\left[\begin{array}{cc}
                                           D_\mathcal A(a)& \\
                                            & D_\mathcal B(b)\\
                                         \end{array}
                                       \right],
\end{equation}
is a bounded module derivation. Furthermore, $D_{\mathcal
A\mathcal B}$ is inner if and only if $D_\mathcal A$ and
$D_\mathcal B$ are inner.
\end{lemma}

Using Lemmas \ref{po1} and \ref{po2}, and similar to \cite[Theorem 2.1]{po}, we have the following
result (see also \cite[Theorem 2.1]{me}) .

\begin{theorem}\label{1}
Let $\mathcal A$ and $\mathcal B$ be unital Banach algebras, and
$\mathcal M$ be a unital $\mathcal A,\mathcal B$-module. Then
\begin{equation}\label{t1}
    \mathcal{H}_\mathfrak T^1(\mathcal T, \mathcal T^{(2n-1)})\simeq\mathcal{H}_\mathfrak A^1(\mathcal A,\mathcal A^{(2n-1)})\oplus\mathcal{H}_\mathfrak A^1(\mathcal B, \mathcal B^{(2n-1)}).
\end{equation} In particular, the triangular Banach algebra $\mathcal T$ is
$(2n-1)$-weakly module amenable \emph{(}as an $\mathfrak
T$-bimodule\emph{)} if and only if $\mathcal A$ and $\mathcal B$
are $(2n-1)$-weakly module amenable \emph{(}as Banach $\mathfrak
A$-bimodules\emph{)}.
\end{theorem}
A Banach $\mathcal A,\mathcal B$-module $\mathcal M$ is said to be
non-degenerate if $\mathcal A m=\{0\}$ implies that $m=0$, and
$m\mathcal B=\{0\}$ implies that $m=0$ for every $m\in\mathcal
M$. If the Banach algebras $\mathcal A$ and $\mathcal B$ have
bounded approximate identity and $\mathcal M$ is essential, then
$\mathcal M$ is a non-degenerated $\mathcal A,\mathcal B$-module.
Also when  $\mathcal M$ is essential, then  $\mathcal M^*$ is a
non-degenerate Banach $\mathcal A,\mathcal B$-module. Although in
the following Proposition $\mathcal A$ and $\mathcal B$ are not
unital but still the conclusion of Lemma \ref{po1} holds. In fact we
obtain the same result with different conditions.

\begin{prop}\label{prop1}
Let $\mathcal A$ and $\mathcal B$ be Banach algebras, and
$\mathcal M$ be Banach $\mathcal A,\mathcal B$-module. Suppose
that $\mathcal A$ possess a bounded approximate identity,
$\mathcal A^{(2n-1)}, \mathcal B^{(2n-1)}$ and $\mathcal
M^{(2n-1)}$ are non-degenerate. Then the $\mathfrak T$-module map
$D:\mathcal T\longrightarrow\mathcal T^{(2n-1)}$ is module
derivation if and only if there exist module derivations
$D_\mathcal A:\mathcal A\longrightarrow\mathcal A^{(2n-1)}$,
$D_\mathcal B:\mathcal B\longrightarrow\mathcal B^{(2n-1)}$ and
$\gamma\in\mathcal M^{(2n-1)}$ such that
                                       \begin{equation}\label{p}
                                        D\Big{(}\left[\begin{array}{cc}
                                           a & m \\
                                            &  b\\
                                         \end{array}
                                       \right]\Big{)}=\left[\begin{array}{cc}
                                           D_\mathcal A(a)-m\gamma & \gamma a-b\gamma \\
                                            & D_\mathcal B(b)+\gamma m\\
                                         \end{array}
                                       \right],
                                       \end{equation}
for every $\left[\begin{array}{cc}
                                           a & m \\
                                            &  b\\
                                         \end{array}
                                       \right]\in\mathcal T$.
\end{prop}
\begin{proof}
Let $D:\mathcal T\longrightarrow\mathcal T^{(2n-1)}$ be a
continuous module derivation. Define $D_\mathcal A:\mathcal
A\longrightarrow\mathcal A^{(2n-1)}$ by $D_\mathcal
A(a)=\pi_\mathcal A(D\Big{(}\left[\begin{array}{cc}
                                           a & 0 \\
                                            &  0\\
                                         \end{array}
                                       \right]\Big{)})$, and $D_\mathcal B:\mathcal B\longrightarrow\mathcal B^{(2n-1)}$ via $D_\mathcal B(b)=\pi_\mathcal B(D\Big{(}\left[\begin{array}{cc}
                                           0 & 0 \\
                                            &  b\\
                                         \end{array}
                                       \right]\Big{)})$. Obviously these maps are $\mathfrak A$-module maps, and are module derivations by \cite[Lemma 2.3]{me}.
 Let $(e_\alpha)_{\alpha\in\Lambda}$ be a bounded approximate identity of $\mathcal A$, and let $D\Big{(}\left[\begin{array}{cc}
                                           a & 0 \\
                                            &  0\\
                                         \end{array}
                                       \right]\Big{)}=\left[\begin{array}{cc}
                                           D_\mathcal A(a) & \eta\\
                                            &  \mu\\
                                         \end{array}
                                       \right]$ for an arbitrary and fixed $a\in\mathcal A$. Then
                                       \begin{eqnarray}\label{4pp1}
                                       \nonumber
                                        D\Big{(}\left[\begin{array}{cc}
                                           a & 0 \\
                                            &  0\\
                                         \end{array}
                                       \right]\Big{)}&=& D\Big{(}\left[\begin{array}{cc}
                                           \lim_\alpha e_\alpha a& 0 \\
                                            &  0\\
                                         \end{array}
                                       \right]\Big{)}=D\Big{(}\lim_\alpha\left[\begin{array}{cc}
                                            e_\alpha & 0 \\
                                            &  0\\
                                         \end{array}
                                       \right]\left[\begin{array}{cc}
                                         a  & 0 \\
                                            &  0\\
                                         \end{array}
                                       \right]\Big{)} \\
                                          \nonumber
                                          &=& \lim_\alpha\Big{(}\left[\begin{array}{cc}
                                            e_\alpha & 0 \\
                                            &  0\\
                                         \end{array}
                                       \right]\cdot\left[\begin{array}{cc}
                                           D_\mathcal A(a) & \eta\\
                                            &  \mu\\
                                         \end{array}
                                       \right] \\
                                       &+& \nonumber\left[\begin{array}{cc}
                                           D_\mathcal A(e_\alpha) & \gamma\\
                                            &  \theta\\
                                         \end{array}
                                       \right]\left[\begin{array}{cc}
                                           a & 0 \\
                                            &  0\\
                                         \end{array}
                                       \right]\Big{)}\\
                                          &=&\left[\begin{array}{cc}
                                          \lim_\alpha D_\mathcal A(ae_\alpha) & \gamma a\\
                                            &  0\\
                                         \end{array}
                                       \right]=\left[\begin{array}{cc}
                                           D_\mathcal A(a) & \gamma a\\
                                            &  0\\
                                         \end{array}
                                       \right].
                                       \end{eqnarray}

Similarly, consider $b\in\mathcal B$ such that
$D\Big{(}\left[\begin{array}{cc}
                                           0 & 0 \\
                                            &  b\\
                                         \end{array}
                                       \right]\Big{)}=\left[\begin{array}{cc}
                                           \theta & \eta \\
                                            &  D_\mathcal B(b)\\
                                         \end{array}
                                       \right]$. Since $\mathcal M^{(2n-1)}$ and $\mathcal A^{(2n-1)}$ are non-degenerate (see \cite[Proposition 2.5]{me}), we have
                                       \begin{equation}\label{4pp2}
                                        D\Big{(}\left[\begin{array}{cc}
                                           0 & 0 \\
                                            &  b\\
                                         \end{array}
                                       \right]\Big{)}=\left[\begin{array}{cc}
                                           0 & -b\gamma \\
                                            &  D_\mathcal B(b)\\
                                         \end{array}
                                       \right].
                                       \end{equation}

Also for every $m\in\mathcal M$ we have

\begin{equation}\label{4pp3}
                                        D\Big{(}\left[\begin{array}{cc}
                                           0 & m \\
                                            &  0\\
                                         \end{array}
                                       \right]\Big{)}=\left[\begin{array}{cc}
                                           -m\gamma & 0 \\
                                            &  \gamma m\\
                                         \end{array}
                                       \right].
                                       \end{equation}

                                       Now, from (\ref{4pp1}), (\ref{4pp2}) and (\ref{4pp3}), we get (\ref{p}), and this completes the proof.
\end{proof}
Now we are ready to prove the main theorem of this section  (see also the proof of \cite[Theorem 2.1]{po}).
\begin{theorem}\label{t2}
Let $\mathcal A$ and $\mathcal B$ both have bounded approximate
identity, and let $\mathcal M$ be non-degenerate. Then for every
$n\geq1$ we have
\begin{equation}\label{tt1}
    \mathcal{H}_\mathfrak T^1(\mathcal T, \mathcal T^{(2n-1)})\simeq\mathcal{H}_\mathfrak A^1(\mathcal A,\mathcal A^{(2n-1)})\oplus\mathcal{H}_\mathfrak A^1(\mathcal B, \mathcal B^{(2n-1)}).
\end{equation}

Furthermore, the corresponding triangular Banach algebra $\mathcal
T\left[\begin{array}{cc}
                                           \mathcal A & \mathcal M \\
                                            &  \mathcal B\\
                                         \end{array}
                                       \right]$ is
$(2n-1)$-weakly module amenable \emph{(}as an $\mathfrak T
$-bimodule\emph{)} if and only if $\mathcal A$ and $\mathcal B$
are $(2n-1)$-weakly module amenable \emph{(}as Banach $\mathfrak A
$-bimodules\emph{)}.
\end{theorem}
\begin{proof}
Suppose that $D:\mathcal T\longrightarrow\mathcal T^{(2n-1)}$ is a
continuous module derivation.  Proposition \ref{prop1} shows that
there are continuous module derivations $D_\mathcal A:\mathcal
A\longrightarrow\mathcal A^{(2n-1)}$, $D_\mathcal B:\mathcal
B\longrightarrow\mathcal B^{(2n-1)}$ and $\gamma\in\mathcal
M^{(2n-1)}$ such that
                                       \begin{equation}\label{tt2}
                                        D\Big{(}\left[\begin{array}{cc}
                                           a & m \\
                                            &  b\\
                                         \end{array}
                                       \right]\Big{)}=\left[\begin{array}{cc}
                                           D_\mathcal A(a)-m\gamma & \gamma a-b\gamma \\
                                            & D_\mathcal B(b)+\gamma m\\
                                         \end{array}
                                       \right],
                                       \end{equation}
for every $\left[\begin{array}{cc}
                                           a & m \\
                                            &  b\\
                                         \end{array}
                                       \right]\in\mathcal T$. Define $\Psi:\mathcal{Z}_\mathfrak T^1(\mathcal T,\mathcal T^{(2n-1)})\longrightarrow\mathcal{H}_\mathfrak A^1(\mathcal A,\mathcal A^{(2n-1)})\oplus\mathcal{H}_\mathfrak T^1(\mathcal B,\mathcal B^{(2n-1)})$
                                       by
                                       \begin{equation}\label{tt3}
                                        \Psi(D)=(D_\mathcal A+\mathcal{N}_\mathfrak A^1(\mathcal A,\mathcal A^{(2n-1)}),D_\mathcal B+\mathcal{N}_\mathfrak A^1(\mathcal B,\mathcal B^{(2n-1)})).
                                       \end{equation}

          Let's show that $\Psi$ is a linear map. If $(e_\alpha)$ is a bounded approximate identity for $\mathcal A$, then for given $\lambda\in\mathbb{C}$ we have
          \begin{eqnarray}\label{tt4}
          \nonumber
            \langle \lambda a', D_\mathcal A(a)\rangle &=&  \langle \lim_\alpha\lambda a'e_\alpha, D_\mathcal A(a)\rangle=\lim_\alpha\langle a',\lambda e_\alpha D_\mathcal A(a)\rangle\\
            &=&\langle \lim_\alpha a' e_\alpha, \lambda D_\mathcal A(a)\rangle
             = \langle a', \lambda D_\mathcal A(a)\rangle\hspace{.4cm}(a,a'\in\mathcal A)
          \end{eqnarray}
          and similarly we have
          \begin{equation}\label{tt5}
             \langle \lambda b', D_\mathcal B(b)\rangle= \langle  b', \lambda D_\mathcal B(b)\rangle\hspace{1cm}(b,b'\in\mathcal B).
          \end{equation}

Now, by applying relations (\ref{tt4}) and (\ref{tt5}) for every
$T_1=\left[\begin{array}{cc}
                                           a_1 & m_1 \\
                                            &  b_1\\
                                         \end{array}
                                       \right]$ and $T_2=\left[\begin{array}{cc}
                                           a_2 & m_2 \\
                                            &  b_2\\
                                         \end{array}
                                       \right]$ in $\mathcal T$ we
                                       get
\begin{eqnarray}
 \nonumber
  \langle (\lambda T_2),D(T_1)\rangle &=& D\Big{(}\left[\begin{array}{cc}
                                           a_1 & m_1 \\
                                            &  b_1\\
                                         \end{array}
                                       \right]\Big{)} \Big{(}\left[\begin{array}{cc}
                                           \lambda a_2 & \lambda m_2 \\
                                            & \lambda b_2\\
                                         \end{array}
                                       \right]\Big{)}\\
   \nonumber
   &=&  \left[\begin{array}{cc}
                                           D_\mathcal A(a_1)-m_1\gamma & \gamma a_1-b_1\gamma \\
                                            & D_\mathcal B(b_1)+\gamma m_1\\
                                         \end{array}
                                       \right]\Big{(}\left[\begin{array}{cc}
                                           \lambda a_2 & \lambda m_2 \\
                                            & \lambda b_2\\
                                         \end{array}
                                       \right]\Big{)}\\
   \nonumber
   &=& D_\mathcal A(a_1)(\lambda a_2)-m_1\gamma(\lambda a_2)+a_1\gamma(\lambda b_2)\\
   \nonumber
   &-& b_1\gamma(\lambda m_2)+D_\mathcal B(b)(\lambda b_2)+m\gamma(\lambda b_2)\\
   \nonumber
   &=& \lambda D_\mathcal A(a_1)(a_2)-\lambda m_1\gamma(a_2)+\lambda a_1\gamma(b_2)\\
   \nonumber
   &-& \lambda b_1\gamma(m_2)+\lambda D_\mathcal B(b)(b_2)+\lambda m\gamma(b_2) \\
     \nonumber
   &=&  \left[\begin{array}{cc}
                                         \lambda  D_\mathcal A(a_1)-\lambda m_1\gamma & \lambda\gamma a_1-\lambda b_1\gamma \\
                                            & \lambda D_\mathcal B(b_1)+\lambda\gamma m_1\\
                                         \end{array}
                                       \right]\Big{(}\left[\begin{array}{cc}
                                            a_2 & m_2 \\
                                            &  b_2\\
                                         \end{array}
                                       \right]\Big{)}\\
                                       &=&  \lambda D(T_1)(T_2).
\end{eqnarray}

Thus $\Psi(\lambda D)=\lambda\Psi(D)$. Obviously,
$\Psi(D_1+D_2)=\Psi(D_1)+\Psi(D_2)$ for all
$D_1,D_2\in\mathcal{Z}_\mathfrak T^1(\mathcal T,\mathcal
T^{(2n-1)})$. Hence, $\Psi$ is a linear map. Now, assume that
$D_\mathcal A\in\mathcal{Z}_\mathfrak A^1(\mathcal A,\mathcal
A^{(2n-1)})$ and $D_\mathcal B\in\mathcal{Z}_\mathfrak
A^1(\mathcal B,\mathcal B^{(2n-1)})$. Then Lemma \ref{po2} implies
that there exists a module derivation $D_{\mathcal A\mathcal B}$
such that
\begin{equation*}
    \Psi(D_{\mathcal A\mathcal B})=(D_\mathcal A+\mathcal{N}_\mathfrak A^1(\mathcal A,\mathcal A^{(2n-1)}),D_\mathcal B+\mathcal{N}_\mathfrak A^1(\mathcal B,\mathcal B^{(2n-1)})),
\end{equation*}
and this indicates that $\Psi$ is surjective. If $D\in\ker\Psi$,
then $\Psi(D)=0$ and thus $D_\mathcal A\in\mathcal{N}_\mathfrak
A^1(\mathcal A,\mathcal A^{(2n-1)})$ and $D_\mathcal
B\in\mathcal{N}_\mathfrak A^1(\mathcal B,\mathcal B^{(2n-1)})$.
Since $D_\mathcal A$ and $D_\mathcal B$ are inner, $D_{\mathcal
A\mathcal B}$ is an inner module derivation by Lemma \ref{po2}.
Therefore we can write
\begin{equation*}
    \left[\begin{array}{cc}
                                           D_\mathcal A(a)-m\gamma & \gamma a-b\gamma \\
                                            & D_\mathcal B(b)+\gamma m\\
                                         \end{array}
                                       \right]=\left[\begin{array}{cc}
                                           D_\mathcal A(a) &  \\
                                            & D_\mathcal B(b)\\
                                         \end{array}
                                       \right]+\left[\begin{array}{cc}
                                          - m\gamma & \gamma a-b\gamma \\
                                            & \gamma m\\
                                         \end{array}
                                       \right].
\end{equation*}

The above equality shows that $D$ is inner. Thus,
$\ker\Psi\subseteq\mathcal{N}_\mathfrak T^1(\mathcal T,\mathcal
T^{(2n-1)})$. On the other hand, if $D\in\mathcal{N}_\mathfrak T
^1(\mathcal T,\mathcal T^{(2n-1)})$, then $D_\mathcal A:\mathcal
A\longrightarrow\mathcal A^{(2n-1)}$ defined by $D_\mathcal
A(a)=\pi_\mathcal A(D\Big{(}\left[\begin{array}{cc}
                                            a & 0 \\
                                            &  0\\
                                         \end{array}
                                       \right]\Big{)})$ and $D_\mathcal B:\mathcal B\longrightarrow\mathcal B^{(2n-1)}$ defined by $D_\mathcal B(b)=\pi_\mathcal B(D\Big{(}\left[\begin{array}{cc}
                                            0 & 0 \\
                                            &  b\\
                                         \end{array}
                                       \right]\Big{)})$ are inner module derivations. Therefore $\Psi(D)=0$, and so $\mathcal{N}_\mathfrak T^1(\mathcal T,\mathcal T^{(2n-1)})\subseteq\ker\Psi$. Finally, we have
\begin{eqnarray*}
    \mathcal{H}_\mathfrak T^1(\mathcal T,\mathcal T^{(2n-1)})&=&  \mathcal{Z}_\mathfrak T^1(\mathcal T,\mathcal T^{(2n-1)})/
    \mathcal{N}_\mathfrak T^1(\mathcal T,\mathcal T^{(2n-1)})\\
    &=&\mathcal{Z}_\mathfrak T^1(\mathcal T,\mathcal T^{(2n-1)})/\ker\Psi\\
   &\simeq&  \emph{\emph{Im}}\Psi=\mathcal{H}_\mathfrak A^1(\mathcal A,\mathcal A^{(2n-1)})\oplus\mathcal{H}_\mathfrak A^1(\mathcal B, \mathcal B^{(2n-1)}).
\end{eqnarray*}
\end{proof}

\section{(2n)-Weak Module Amenability}

As it is seen in the previous section, $(2n-1)$-weak module
amenability of a triangular Banach algebra $\mathcal T$ depends
on $(2n-1)$-weak module amenability of Banach algebras
$\mathcal A$ and $\mathcal B$ while this fails to be true in
the even case in general. We need the following lemma which is analogous to
\cite[Proposition 3.9]{fo2} in the module case. We include the
proof.
\begin{lemma}\label{l2.1}
Let $D:\mathcal T\longrightarrow\mathcal T^{(2n)}$ be a
continuous module derivation. Then there exist
$\gamma\in\mathcal{M}^{(2n)}$, continuous module derivations
$D_\mathcal A:\mathcal{A}\longrightarrow\mathcal A^{(2n)}$,
$D_\mathcal B:\mathcal{B}\longrightarrow \mathcal B^{(2n)}$ and a
continuous $\mathfrak A$-module map $\rho
:\mathcal{M}\longrightarrow \mathcal{M}^{(2n)}$ such that
\begin{itemize}
  \item[(i)]$D\Big{(} \left[
         \begin{array}{rr}
              a & 0 \\
                & 0
          \end{array} \right]\Big{)}=\left[
         \begin{array}{rr}
              D_\mathcal A(a) & a\cdot \gamma \\
                & 0
          \end{array} \right]$,\quad  \emph{(}$a\in\mathcal{A}$\emph{)};
  \item[(ii)]$D \Big{(}\left[
         \begin{array}{rr}
              0 & 0 \\
                & b
          \end{array} \right]\Big{)}=\left[
         \begin{array}{rr}
              0 & -\gamma \cdot b \\
                & D_\mathcal B(b)
          \end{array} \right]$,\quad  \emph{(}$b\in\mathcal{B}$\emph{)};
  \item[(iii)]$D\Big{(}\left[
         \begin{array}{rr}
             0  & m \\
                & 0
          \end{array} \right]\Big{)}=\left[
         \begin{array}{rr}
             0  & \rho(m) \\
                & 0
          \end{array} \right]$, \quad \emph{(}$m\in\mathcal{M}$\emph{)};
     \item[(iv)]$\rho(a\cdot m)=D_\mathcal A(a)\cdot m+a\cdot \rho(m)$, \quad  \emph{(}$a\in\mathcal{A}, m\in\mathcal{M}$\emph{)};
     \item[(v)]$\rho(m\cdot b)=\rho(m)\cdot b+m\cdot D_\mathcal B(b)$ \quad  \emph{(}$a\in\mathcal{A}, m\in\mathcal{M}$\emph{)};
     \item[(vi)]If $D_\mathcal A:\mathcal{A}\longrightarrow\mathcal A^{(2n)}$ and $D_\mathcal B:\mathcal{B}\longrightarrow\mathcal B^{(2n)}$
     are continuous module derivations and $\rho_\mathcal{M}:\mathcal{M}\longrightarrow\mathcal{M}^{(2n)}$ is a continuous $\mathfrak A$-module map
     that satisfies $(\emph{iv})$ and $(\emph{v})$. Then $D:\mathcal T\rightarrow\mathcal T^{(2n)}$ defined by $ \left[
         \begin{array}{rr}
             a  & m \\
                & b
          \end{array} \right]\longmapsto\left[
         \begin{array}{rr}
            D_\mathcal A(a)  & \rho_\mathcal{M}(m) \\
                & D_\mathcal B(b)
          \end{array} \right]$ is a continuous module derivation.

\end{itemize}
\end{lemma}
\begin{proof}
In the light of  \cite[Proposition 3.9]{fo2}, we just show that
these maps are module maps. Let $D$ be a $\mathfrak T$-module
map. Then
\begin{eqnarray*}
   \langle a_2,\alpha\cdot D_\mathcal A(a_1)\rangle&=& \langle a_2\cdot\alpha,D_\mathcal A(a_1)\rangle=\langle a_2\cdot\alpha,D_\mathcal A(a_1)\rangle+\langle 0, a_1\cdot\gamma\rangle\\
   &=& \langle \left[
         \begin{array}{rr}
              a_2\cdot\alpha & 0 \\
                & 0
          \end{array} \right], \left[
         \begin{array}{rr}
              D_\mathcal A(a_1) & a_1\cdot\rho\\
                & 0
          \end{array} \right]\rangle \\
             &=& \langle \left[
         \begin{array}{rr}
              a_2 & 0 \\
                & 0
          \end{array} \right]\cdot\left[
         \begin{array}{rr}
              \alpha & 0 \\
                & \alpha
          \end{array} \right],D\Big{(} \left[
         \begin{array}{rr}
              a_1 & 0 \\
                & 0
          \end{array} \right]\Big{)}\rangle \\
                &=&  \langle \left[
         \begin{array}{rr}
              a_2 & 0 \\
                & 0
          \end{array} \right],D\Big{(} \left[
         \begin{array}{rr}
              \alpha & 0 \\
                & \alpha
          \end{array} \right]\cdot\left[
         \begin{array}{rr}
              a_1 & 0 \\
                & 0
          \end{array} \right]\Big{)}\rangle\\
             &=&  \langle \left[
         \begin{array}{rr}
              a_2 & 0 \\
                & 0
          \end{array} \right],D\Big{(}\left[
         \begin{array}{rr}
             \alpha\cdot a_1 & 0 \\
                & 0
          \end{array} \right]\Big{)}\rangle\\
             &=& \langle \left[
         \begin{array}{rr}
              a_2 & 0 \\
                & 0
          \end{array} \right], \left[
         \begin{array}{rr}
              D_\mathcal A(\alpha a_1) & \alpha\cdot a_1\cdot\rho\\
                & 0
          \end{array} \right]\rangle\\
             &=& \langle a_2,D_\mathcal A(\alpha\cdot a_1)\rangle+\langle 0, \alpha\cdot a_1\cdot\gamma\rangle=\langle a_2, D_\mathcal A(\alpha\cdot a_1)\rangle,
\end{eqnarray*}
for all $a_1,a_2\in\mathcal A$ and $\alpha\in\mathfrak A$. This
means that $D_\mathcal A(\alpha\cdot a)=\alpha\cdot D_\mathcal
A(a)$, for all $a\in\mathcal A$ and $\alpha\in\mathfrak A$. Similarly we can show that $D_\mathcal
A(a\cdot\alpha)=D_\mathcal A(a)\cdot\alpha$ and
$D_\mathcal B$ and $\rho$ are $\mathfrak A$-module maps.
Therefore the assertions (i)-(v) hold. For (vi), suppose that
$D_\mathcal A$, $D_\mathcal B$ and $\rho_\mathcal M$ are
$\mathfrak A$-module maps. We show that $D$ is a $\mathfrak
T$-module map. Given $T_1=\left[\begin{array}{cc}
                                           a_1 & m_1 \\
                                            &  b_1\\
                                         \end{array}
                                       \right], T_2=\left[\begin{array}{cc}
                                           a_2 & m_2 \\
                                            &  b_2\\
                                         \end{array}
                                       \right]\in\mathcal T$ and $\Upsilon=\left[
         \begin{array}{rr}
              \alpha & 0 \\
                & \alpha
          \end{array} \right]\in\mathfrak T$, we have
          \begin{eqnarray*}
            \langle T_2,\Upsilon\cdot D(T_1)\rangle &=& \langle T_2\cdot \Upsilon,  D(T_1)\rangle \\
            &=& \langle \left[\begin{array}{cc}
                                           a_2\cdot\alpha & m_2\cdot\alpha \\
                                            &  b_2\cdot\alpha\\
                                         \end{array}
                                       \right], \left[
         \begin{array}{rr}
            D_\mathcal A(a_1)  & \rho_\mathcal{M}(m_1) \\
                & D_\mathcal B(b_1)
          \end{array} \right]\rangle\\
             &=&  \langle  a_2\cdot\alpha,  D_\mathcal A(a_1)\rangle+\langle m_2\cdot\alpha, \rho_\mathcal{M}(m_1)\rangle+\langle b_2\cdot\alpha, D_\mathcal B(b_1)\rangle\\
             &=&   \langle  a_2,  D_\mathcal A(\alpha\cdot a_1)\rangle+\langle m_2, \rho_\mathcal{M}(\alpha\cdot m_1)\rangle+\langle b_2, D_\mathcal B(\alpha\cdot b_1)\rangle\\
             &=&  \langle \left[\begin{array}{cc}
                                           a_2 & m_2 \\
                                            &  b_2\\
                                         \end{array}
                                       \right], \left[
         \begin{array}{rr}
            D_\mathcal A(\alpha\cdot a_1)  & \rho_\mathcal{M}(\alpha\cdot m_1) \\
                & D_\mathcal B(\alpha\cdot b_1)
          \end{array} \right]\rangle\\
             &=& \langle \left[\begin{array}{cc}
                                           a_2 & m_2 \\
                                            &  b_2\\
                                         \end{array}
                                       \right], D\Big{(} \left[
         \begin{array}{rr}
              \alpha & 0 \\
                & \alpha
          \end{array} \right]\cdot\left[
         \begin{array}{rr}
              a_1 & m_1 \\
                & b_1
          \end{array} \right]\Big{)}\rangle \\
          &=& \langle T_2, D(\Upsilon\cdot T_1)\rangle.
          \end{eqnarray*}

Therefore $D(\Upsilon\cdot T_1)=\Upsilon\cdot D(T_1)$. Similarly we obtain
$D(T_1\cdot\Upsilon)= D(T_1)\cdot\Upsilon$. Thus, $D$ is a
$\mathfrak T$-module map.
\end{proof}
The following sets which are used in this
section are introduced in \cite{fo2}. For each positive
integer $n$, we define the centralizer of $\mathcal{A}$ in
$\mathcal A^{(2n)}$ by $Z_\mathcal{A}(\mathcal
A^{(2n)})=\{x\in\mathcal A^{(2n)}: x\cdot a=a\cdot x$ for all
$a\in\mathcal{A}\}$ and similarly,
$Z_\mathcal{B}(\mathcal{B}^{(2n)})=\{z\in\mathcal{B}^{(2n)}:
z\cdot b=b\cdot z$ for all $b\in\mathcal{B}\}$. The elements of
$$ZR_{\mathfrak A,\mathcal{A},\mathcal{B}}(\mathcal{M},\mathcal{M}^{(2n)})=\{\rho_{x,z}:\mathcal{M}\rightarrow\mathcal{M}^{(2n)}: x\in Z_\mathcal{A}(\mathcal{A}^{(2n)}), z\in Z_\mathcal{B}(\mathcal{B}^{(2n)})\}$$
are called {\it central Rosenblum operators} on $\mathcal{M}$ with
coefficients in $\mathcal{M}^{(2n)}$ in which
$\rho_{x,z}(m)=x\cdot m-m\cdot z$ is an $\mathfrak A$-module map.
We also define the following set
$$  {\emph{\emph{Hom}}}_{\mathfrak A,\mathcal{A},\mathcal{B}}(\mathcal{M},\mathcal{M}^{(2n)}) = \{\varphi:\mathcal{M}\rightarrow\mathcal{M}^{(2n)}: \varphi(a\cdot m)=a\cdot \varphi(m), ~\varphi(m\cdot b)=\varphi(m)\cdot b,$$ 
$$\varphi(\alpha\cdot m)=\alpha\cdot \varphi(m), ~\varphi(m\cdot \alpha)=\varphi(m)\cdot \alpha , ~~{\emph{\emph{for}}}~~ {\emph{\emph{all}}}~~ \alpha\in\mathfrak A, a\in \mathcal{A}, m\in\mathcal{M}, b\in\mathcal{B} \}.$$

The following theorem is analogous to
Lemma 3.11 and Theorem 3.12 from \cite{fo2} in a more general setting (part (iii) is a module version of \cite[Theorem 3.2]{me}).

\begin{theorem}\label{T5.2}
Let $\mathcal{A}$ and $\mathcal{B}$ be unital Banach algebras and
$\mathcal{M}$ be a unital Banach
$\mathcal{A},\mathcal{B}$-module. Then we have
\begin{itemize}
  \item[(i)] $ZR_{\mathfrak A,\mathcal A,\mathcal B}(\mathcal M,\mathcal M^{(2n)})\subseteq \emph{Hom}_{\mathfrak A,\mathcal{A},\mathcal{B}}(\mathcal M,\mathcal M^{(2n)}),$       
    \item[(ii)]If $\Omega\in \emph{Hom}_{\mathfrak T,\mathcal{A},\mathcal{B}}\Big{(}\left[
                              \begin{array}{rr}
                                  0 & \mathcal M \\
                                     & 0
                               \end{array} \right],\left[
                              \begin{array}{rr}
                                  0 & \mathcal M ^{(2n)}\\
                                     & 0
                               \end{array} \right]\Big{)},$ then the map
$$\Delta_\Omega\left[
         \begin{array}{rr}
             a  & m \\
                & b
          \end{array} \right]=\Omega\Big{(}\left[
         \begin{array}{rr}
            0  & m\\
                & 0
          \end{array} \right]\Big{)}=\left[
           \begin{array}{rr}
           0  & \Omega_\mathcal M(m)\\
             & 0
             \end{array} \right]\in Z_\mathfrak T^1(\mathcal T,\mathcal T^{(2n)}),$$ where $\Omega_\mathcal M\in \emph{Hom}_{\mathfrak A,\mathcal{A},\mathcal{B}}(\mathcal{M},\mathcal{M}^{(2n)})$.
 Furthermore $\Delta_\Omega$ is inner if and only if $\Omega$ is a central Rosenblum operator on $\mathcal M $ with coefficients
 in $\mathcal M^{(2n)}$.
\item[(iii)]If $\mathcal{A}$ and $\mathcal{B}$ are $(2n)$-weakly module amenable as $\mathfrak T$-modules, then
$$H_\mathfrak T^1(\mathcal T,\mathcal T^{(2n)})\simeq  \emph{Hom}_{\mathfrak T,\mathcal{A},\mathcal{B}}\Big{(}\left[
                               \begin{array}{rr}
                                   0 & \mathcal M \\
                                      & 0
                                \end{array} \right],\left[
                               \begin{array}{rr}
                                   0 & \mathcal M ^{(2n)}\\
                                      & 0
                                \end{array} \right]\Big{)}\Big{/}$$ $$ZR_{\mathfrak T,\mathcal{A},\mathcal{B}}\Big{(}\left[ \begin{array}{rr}
                                                                     0 & \mathcal M \\
                                                                        & 0
                                                                  \end{array} \right], \left[
                                                                  \begin{array}{rr}
                                                                                           0 & \mathcal M ^{(2n)}\\
                                                                                              & 0
                                                                                        \end{array} \right]\Big{)}.$$ In the other words,
$$H_\mathfrak T^1(\mathcal T,\mathcal T^{(2n)})\simeq  \emph{Hom}_{\mathfrak A,\mathcal{A},\mathcal{B}}(\mathcal{M},\mathcal{M}^{(2n)})\big{/}ZR_{\mathfrak A,\mathcal{A},\mathcal{B}}(\mathcal{M},\mathcal{M}^{(2n)}).$$
\end{itemize}
\end{theorem}
\begin{proof}
For statements (i) and (ii),   we only need to show that
$\Delta_\varphi$ is an $\mathfrak T$-module map. The rest of the
proof is the same as \cite[Lemma 3.11]{fo2}. For given
$T=\left[\begin{array}{cc}
                                           a & m \\
                                            &  b\\
                                         \end{array}
                                       \right]\in\mathcal T$ and $\Upsilon=\left[
         \begin{array}{rr}
              \alpha & 0 \\
                & \alpha
          \end{array} \right]\in\mathfrak T$, we have
           \begin{eqnarray*}
          \Delta_\Omega\big{(}\Upsilon\cdot T\big{)} &=& \Delta_\Omega\Big{(}\left[
         \begin{array}{rr}
              \alpha & 0 \\
                & \alpha
          \end{array} \right]\cdot\left[\begin{array}{cc}
                                           a & m \\
                                            &  b\\
                                         \end{array}
                                       \right]\Big{)}=\Delta_\Omega\Big{(}\left[\begin{array}{cc}
                                           \alpha\cdot a &\alpha\cdot m \\
                                            & \alpha\cdot b\\
                                         \end{array}
                                       \right]\Big{)} \\
             &=&  \left[\begin{array}{cc}
                                           0 & \alpha\cdot \Omega_\mathcal M(m) \\
                                            &  0\\
                                         \end{array}
                                       \right]=\left[
         \begin{array}{rr}
              \alpha & 0 \\
                & \alpha
          \end{array} \right]\cdot\left[\begin{array}{cc}
                                           0 &  \Omega_\mathcal M(m) \\
                                            &  0\\
                                         \end{array}
                                       \right]=\Upsilon\cdot\Delta_\Omega(T),
          \end{eqnarray*}
  and similarly $\Delta_\Omega(T\cdot\Upsilon)=\Delta_\Omega(T)\cdot\Upsilon$.

           For (iii), the proof is similar to \cite[Theorem 3.12]{fo2} but we give the proof for the sake of completeness. Consider
           $$\mathfrak{F}:{\text{Hom}}_{\mathfrak T,\mathcal{A},\mathcal{B}}\Big{(}\left[
         \begin{array}{rr}
             0 & \mathcal M \\
                & 0
          \end{array} \right],\left[
         \begin{array}{rr}
             0 & \mathcal M ^{(2n)}\\
                & 0
          \end{array} \right]\Big{)}\rightarrow H_\mathfrak T^1(\mathcal T,\mathcal T^{(2n)})$$ defined by $\Omega\mapsto \overline{\Delta}_\Omega$, where
          $\overline{{\Delta}}_\Omega$ denotes the equivalence class of ${\Delta}_\Omega$ in $H_\mathfrak T^1(\mathcal T,\mathcal T^{(2n)})$.
           It follows from (ii) that 
          $\mathfrak{F}$ is an $\mathfrak T$-module map. Let $D:\mathcal T\longrightarrow\mathcal T^{(2n)}$ be a module derivation. By Lemma \ref{l2.1} there are module
          derivations $D_\mathcal A:\mathcal A\longrightarrow\mathcal A^{(2n)}$, $D_\mathcal B:\mathcal B\longrightarrow\mathcal B^{(2n)}$, and a $\mathfrak A$-module map $\rho:\mathcal M\longrightarrow\mathcal M^{(2n)}$ and
          an element $\gamma\in\mathcal M^{(2n)}$ such that

          \begin{equation}\label{e2}
            D\Big{(}\left[\begin{array}{cc}
                                           a & m \\
                                            &  b\\
                                         \end{array}
                                       \right]\Big{)}=\left[\begin{array}{cc}
                                           D_\mathcal A(a) &a\cdot\gamma-\gamma\cdot b+\rho(m) \\
                                            & D_\mathcal B(b)\\
                                         \end{array}
                                       \right].
          \end{equation}

          Since $\mathcal A$ and $\mathcal B$ both are $(2n)$-weakly module amenable, there exist $x\in\mathcal A$ and $y\in\mathcal B$
          such that $D_\mathcal A(a)=a\cdot x-x\cdot a=D_x(a)$ and $D_\mathcal B(b)=b\cdot y-y\cdot b=D_y(b)$. Define $D_0:\mathcal T\longrightarrow\mathcal T^{(2n)}$ as follows:
          \begin{equation}\label{e3}
            D_0\Big{(}\left[\begin{array}{cc}
                                           a & m \\
                                            &  b\\
                                         \end{array}
                                       \right]\Big{)}=\left[\begin{array}{cc}
                                           D_x(a) &a\cdot\gamma-\gamma\cdot b-\rho_{x,y}(m) \\
                                            & D_y(b)\\
                                         \end{array}
                                       \right].
          \end{equation}

          It is easy to see that $D_0$ is an inner $\mathfrak T$-module derivation induced by $\left[\begin{array}{cc}
                                           x & \gamma \\
                                            &  y\\
                                         \end{array}
                                       \right]$ (note that in the proof of \cite[Theorem 3.12]{fo2}, there is a misprint). Set $D_1=D-D_0$. Thus $\overline{D}_1=\overline{D}$. Then

          \begin{equation}\label{e4}
              D_1\Big{(}\left[\begin{array}{cc}
                                           a & m \\
                                            &  b\\
                                         \end{array}
                                       \right]\Big{)}=\left[\begin{array}{cc}
                                           0 &\rho(m)+\rho_{x,y}(m) \\
                                            & 0\\
                                         \end{array}
                                       \right],
          \end{equation}
          for all $\left[\begin{array}{cc}
                                           a & m \\
                                            &  b\\
                                         \end{array}
                                       \right]\in\mathcal T$. Define $\Omega=\left[\begin{array}{cc}
                                           0 & \rho(m)+\rho_{x,y}(m) \\
                                            &  0\\
                                         \end{array}
                                       \right]$. Clearly $\Omega$ belongs to $$\text{Hom}_{\mathfrak T,\mathcal{A},\mathcal{B}}(\left[
         \begin{array}{rr}
             0 & \mathcal M \\
                & 0
          \end{array} \right],\left[
         \begin{array}{rr}
             0 & \mathcal M ^{(2n)}\\
                & 0
          \end{array} \right]).$$

  By (ii), we have $\mathfrak{F}(\Omega)=\overline{\Delta}_\Omega=\overline{D}_1=\overline{D}$. This means that
          $\mathfrak{F}$ is surjective. Finally, we must show that $\ker\mathfrak{F}=ZR_{\mathcal A,\mathcal B}(\mathcal M,\mathcal M^{(2n)})$. Let $\Omega\in\ker\mathfrak{F}$.
          Then $\mathfrak{F}(\Omega)\Big{(}\left[\begin{array}{cc}
                                                              a & m \\
                                                               &  b\\
                                                            \end{array}
                                                          \right]\Big{)}=\overline{D} \Big{(}\left[\begin{array}{cc}
                                                                                                                a & m \\
                                                                                                                 &  b\\
                                                                                                              \end{array}
                                                                                                            \right]\Big{)}=0=
                                                                                                            \overline{\Delta}_\Omega\Big{(}\left[\begin{array}{cc}
                                                                                                                                                                               a & m \\
                                                                                                                                                                                &  b\\
                                                                                                                                                                             \end{array}
                                                                                                                                                                           \right]\Big{)}$,
  hence  $\overline{\Delta}_\Omega$ is inner. Again by (ii), we have $\Omega\in ZR_{\mathcal A,\mathcal B}(\mathcal M,\mathcal M^{(2n)})$,
  and (i) implies that $\ker\mathfrak{F}=ZR_{\mathcal A,\mathcal B}(\mathcal M,\mathcal M^{(2n)})$.  Therefore
  \begin{eqnarray*}
       H_\mathfrak T^1(\mathcal T,\mathcal T^{(2n)})&\simeq&  \text{Hom}_{\mathfrak T,\mathcal{A},\mathcal{B}}\Big{(}\left[
                                    \begin{array}{rr}
                                        0 & \mathcal M \\
                                           & 0
                                     \end{array} \right],\left[
                                    \begin{array}{rr}
                                        0 & \mathcal M ^{(2n)}\\
                                           & 0
                                     \end{array} \right]\Big{)}\Big{/}\ker\mathfrak{F}\\
                                          &=& \frac{\text{Hom}_{\mathfrak T,\mathcal{A},\mathcal{B}}\Big{(}\left[
                                          \begin{array}{rr}
                                              0 & \mathcal M \\
                                                 & 0
                                           \end{array} \right],\left[
                                          \begin{array}{rr}
                                              0 & \mathcal M ^{(2n)}\\
                                                 & 0
                                           \end{array} \right]\Big{)}}{ZR_{\mathfrak T,\mathcal{A},\mathcal{B}}\Big{(}\left[
                                                                            \begin{array}{rr}
                                                                                0 & \mathcal M \\
                                                                                   & 0
                                                                             \end{array} \right], \left[
                                                                             \begin{array}{rr}
                                                                                                      0 & \mathcal M ^{(2n)}\\
                                                                                                         & 0
                                                                                                   \end{array} \right]\Big{)}}.
  \end{eqnarray*}
\end{proof}
The following result may be proved like
Proposition \ref{prop1} using a modification of Lemma
\ref{l2.1}.
\begin{prop}\label{P5}
Let $\mathcal A$ or $\mathcal{B}$ has a bounded approximate identity, and let $\mathcal A^{(2n)}$, $\mathcal{B}^{(2n)}$ and $\mathcal M^{(2n)}$ be non-degenerate.
Then the $\mathfrak T$-module map
$D:\mathcal T\longrightarrow\mathcal T^{(2n)}$ is a module
derivation if and only if there exist module derivations
$D_\mathcal A:\mathcal A\longrightarrow\mathcal A^{(2n)}$ and
$D_\mathcal B:\mathcal B\longrightarrow\mathcal B^{(2n)}$, and $\mathfrak{A}$-module map $\rho:\mathcal M\longrightarrow\mathcal M^{(2n)}$ which satisfies  conditions
 \emph{(iv)} and \emph{(v)} of Lemma \ref{l2.1} such that
                                       \begin{equation}\label{p1}
                                        D\Big{(}\left[\begin{array}{cc}
                                           a & m \\
                                            &  b\\
                                         \end{array}
                                       \right]\Big{)}=\left[\begin{array}{cc}
                                           D_\mathcal A(a) & a\cdot \gamma -\gamma\cdot b+\rho(m) \\
                                            & D_\mathcal B(b)\\
                                         \end{array}
                                       \right],
                                       \end{equation}
in which $\gamma\in\mathcal M^{(2n)}$ and $\left[\begin{array}{cc}
                                           a & m \\
                                            &  b\\
                                         \end{array}
                                      \right]\in\mathcal T$.
\end{prop}
\begin{proof}
We argue similar to the proof of Proposition \ref{prop1}. Let $D:\mathcal
T\longrightarrow\mathcal T^{(2n)}$ be a continuous module
derivation. Define $D_\mathcal A:\mathcal
A\longrightarrow\mathcal A^{(2n)}$ by $D_\mathcal
A(a)=\pi_\mathcal A(D\Big{(}\left[\begin{array}{cc}
                                           a & 0 \\
                                            &  0\\
                                         \end{array}
                                       \right]\Big{)})$, and $D_\mathcal B:\mathcal B\longrightarrow\mathcal B^{(2n)}$ by $D_\mathcal B(b)=\pi_\mathcal B(D\Big{(}\left[\begin{array}{cc}
                                           0 & 0 \\
                                            &  b\\
                                         \end{array}
                                       \right]\Big{)})$. Obviously these maps are $\mathfrak A$-module maps, and by \cite[Lemma 2.3]{me}, they
are module derivations. Let $(e_\alpha)_{\alpha\in\Lambda}$ be a
bounded approximate identity of $\mathcal A$, and let
$D\Big{(}\left[\begin{array}{cc}
                                           a & 0 \\
                                            &  0\\
                                         \end{array}
                                       \right]\Big{)}=\left[\begin{array}{cc}
                                           D_\mathcal A(a) & \eta\\
                                            &  \mu\\
                                         \end{array}
                                       \right]$ for all $a\in\mathcal A$. Then
                                       \begin{eqnarray}\label{pp1}
                                       \nonumber
                                        D\Big{(}\left[\begin{array}{cc}
                                           a & 0 \\
                                            &  0\\
                                         \end{array}
                                       \right]\Big{)}&=& D\Big{(}\left[\begin{array}{cc}
                                           \lim_\alpha  ae_\alpha& 0 \\
                                            &  0\\
                                         \end{array}
                                       \right]\Big{)}=D\Big{(}\lim_\alpha\left[\begin{array}{cc}
                                         a  & 0 \\
                                            &  0\\
                                         \end{array}
                                       \right]\left[\begin{array}{cc}
                                            e_\alpha & 0 \\
                                            &  0\\
                                         \end{array}
                                       \right]\Big{)} \\
                                          \nonumber
                                          &=& \lim_\alpha\Big{(}\left[\begin{array}{cc}
                                           a & 0 \\
                                            &  0\\
                                         \end{array}
                                       \right]\cdot\left[\begin{array}{cc}
                                           D_\mathcal A(e_\alpha) & \gamma\\
                                            &  \theta\\
                                         \end{array}
                                       \right]\\
                                       \nonumber
                                       &+&\left[\begin{array}{cc}
                                           D_\mathcal A(a) & \eta\\
                                            &  \mu\\
                                         \end{array}
                                       \right]\cdot\left[\begin{array}{cc}
                                            e_\alpha & 0 \\
                                            &  0\\
                                         \end{array}
                                       \right]\Big{)}\\
                                          &=&\left[\begin{array}{cc}
                                          \lim_\alpha D_\mathcal A(ae_\alpha) &  a\cdot\gamma\\
                                            &  0\\
                                         \end{array}
                                       \right]=\left[\begin{array}{cc}
                                           D_\mathcal A(a) & a\cdot\gamma\\
                                            &  0\\
                                         \end{array}
                                       \right].
                                       \end{eqnarray}

Let $b\in\mathcal B$ and $D\Big{(}\left[\begin{array}{cc}
                                           0 & 0 \\
                                            &  b\\
                                         \end{array}
                                       \right]\Big{)}=\left[\begin{array}{cc}
                                           \theta & \eta \\
                                            &  D_\mathcal B(b)\\
                                         \end{array}
                                       \right]$. Since $\mathcal M^{(2n)}$ and $\mathcal A^{(2n)}$ are non-degenerate (see \cite[Theorem 3.2]{me}),
                                       we have
                                       \begin{equation}\label{pp2}
                                        D\Big{(}\left[\begin{array}{cc}
                                           0 & 0 \\
                                            &  b\\
                                         \end{array}
                                       \right]\Big{)}=\left[\begin{array}{cc}
                                           0 & -\gamma\cdot b \\
                                            &  D_\mathcal B(b)\\
                                         \end{array}
                                       \right].
                                       \end{equation}

For $m\in\mathcal M$, suppose that $D\Big{(}\left[\begin{array}{cc}
                                           0 & m \\
                                            &  0\\
                                         \end{array}
                                       \right]\Big{)}=\left[\begin{array}{cc}
                                           \theta & \eta \\
                                            &  \mu\\
                                         \end{array}
                                       \right]$, then by \cite[Theorem 3.2]{me}, $\theta=\mu=0$. Now, we define $\rho:\mathcal M\longrightarrow\mathcal M^{(2n)}$ by $\rho(m)=\pi_\mathcal M(D\Big{(}\left[\begin{array}{cc}
                                           0 & m \\
                                            &  0\\
                                         \end{array}
                                       \right]\Big{)})$. It is clear that $\rho$ is an $\mathfrak A$-module map. Therefore it  suffices to show that
 $\rho$ satisfies in conditions (iv) and (v) of  Lemma \ref{l2.1}. We check these conditions as
 follows:
\begin{eqnarray*}
  \rho(a\cdot m) &=& \pi_\mathcal M(D\Big{(}\left[\begin{array}{cc}
                                           0 & a\cdot m \\
                                            &  0\\
                                         \end{array}
                                       \right]\Big{)})=\pi_\mathcal M(D\Big{(}\left[\begin{array}{cc}
                                           a & 0  \\
                                            &  0\\
                                         \end{array}
                                       \right]\cdot\left[\begin{array}{cc}
                                           0 &  m \\
                                            &  0\\
                                         \end{array}
                                       \right]\Big{)})  \\
   &=&\pi_\mathcal M\Big{(}\left[\begin{array}{cc}
                                           a & 0\\
                                            &  0\\
                                         \end{array}
                                       \right]\cdot\left[\begin{array}{cc}
                                           0 &  \eta \\
                                            &  0\\
                                         \end{array}
                                       \right]+\left[\begin{array}{cc}
                                           D_\mathcal A(a) & a\cdot\gamma\\
                                            &  0\\
                                         \end{array}
                                       \right]\cdot\left[\begin{array}{cc}
                                           0 &  m \\
                                            &  0\\
                                         \end{array}
                                       \right]\Big{)}  \\
   &=&  a\cdot\eta+D_\mathcal A(a)\cdot m= a\cdot\rho(m)+D_\mathcal A(a)\cdot m,
\end{eqnarray*}
and similarly we have $\rho(m\cdot b)=\rho(m)\cdot b+m\cdot D_\mathcal B(b)$.
\end{proof}
We can now rephrase part (iii) of  Theorem \ref{T5.2} as follows.
\begin{theorem}\label{T5.4}
Let $\mathcal A$ or $\mathcal{B}$ has a bounded approximate identity, and let $\mathcal A^{(2n)}$, $\mathcal{B}^{(2n)}$ and $\mathcal M^{(2n)}$ be non-degenerate. If $\mathcal A$ and $\mathcal{B}$ are $(2n)$-weakly module amenable as $\mathfrak{A}$-modules, then
\begin{equation}\label{t5.4}
    H_\mathfrak T^1(\mathcal T,\mathcal T^{(2n)})\simeq  {\emph{Hom}}_{\mathfrak A,\mathcal{A},\mathcal{B}}(\mathcal{M},\mathcal{M}^{(2n)})\big{/}ZR_{\mathfrak A,\mathcal{A},\mathcal{B}}(\mathcal{M},\mathcal{M}^{(2n)}).
\end{equation}
\end{theorem}
\begin{proof} Appling Proposition \ref{P5}, the argument of
Theorem \ref{T5.2} can be repeated to obtain the result.
\end{proof}
\begin{corol}\label{corf}
Let $\mathcal A$ has a bounded approximate identity, and $\mathcal A^{(2n)}$
be non-degenerate. If $\mathcal A$ is $(2n)$-weakly module amenable (as
an $\mathfrak{A}$-module), then
$\mathcal{T}=\left[\begin{array}{cc}
                                           \mathcal A & \mathcal A \\
                                            & \mathcal A\\
                                         \end{array}
                                       \right]$ is $(2n)$-weakly module amenable (as an $\mathfrak{T}$-module).
\end{corol}
\begin{proof}
Let $(e_\alpha)$ be a bounded approximate identity of $\mathcal A$, and
let $\varphi$ in $\text{Hom}_{\mathfrak
A,\mathcal{A},\mathcal{A}}$\\$(\mathcal{A},\mathcal{A}^{(2n)})$. Then
there exists $E\in\mathcal A^{(2n)}$ and a subnet
$\{\varphi(e_\beta)\}$ of $\{\varphi(e_\beta)\}$ such that
$\varphi(e_\beta)\stackrel{w^*}{\longrightarrow}E$. We have
\begin{equation*}
    \varphi(a)=w^*-\lim_\beta\varphi(e_\beta a)=w^*-\lim_\beta e_\beta\varphi(a)=Ea.
\end{equation*}
Similarly, $\varphi(a)=aE$. This shows that $\varphi\in
ZR_{\mathfrak A,\mathcal A,\mathcal A}(\mathcal A,\mathcal A^{(2n)})$ and $
H_\mathfrak T^1(\mathcal T,\mathcal T^{(2n)})$\\$=\{0\}$ by Theorem
\ref{T5.4}.
\end{proof}


\section{Examples}
In this section, by using the results of the previous sections we show that under which conditions the Banach algebra
$\mathcal T_0=\left[\begin{array}{cc}
                                           \ell ^{1}(S) & \ell ^{1}(S) \\
                                            &  \ell ^{1}(S)\\
                                         \end{array}
                                       \right]$ is permanently weakly
module amenable and  module Arens regular where $S$ is an inverse semigroup. 
\begin{definition} A discrete semigroup $S$ is called an inverse
semigroup if for each $ s \in S $ there is a unique element $s^*
\in S$ such that $ss^*s=s$ and $s^*ss^*=s^*$. An element $e \in
S$ is called an idempotent if $e=e^*=e^2$. The set of idempotents
of $S$ is denoted by $E$.\end{definition}

Let $S$ be an inverse semigroup with the set of idempotents $E$.
By \cite[Theorem V.1.2]{how} $E$ is a commutative subsemigroup of
$S$ and a semilattice, $ \ell ^{1}(E)$ could be regarded as
a commutative subalgebra of $ \ell ^{1}(S)$, and therefore $\ell
^{1}(S)$ is a Banach algebra and a Banach $ \ell ^{1}(E)$-module
with compatible actions \cite{am1}.

Let $k\in \mathbb{N}$. Recall that $E$ satisfies condition $D_k$
\cite{dun} if given $f_1,f_2,...,f_{k+1}\in E$ there exist $e\in
E$ and $i,j$ such that
$$1\leq i<j\leq k+1, f_ie=f_i, f_je=f_j.$$

Duncan and Namioka in \cite[Theorem 16]{dun} proved that for any
inverse semigroup $S$, $ \ell ^{1}(S)$ has a bounded approximate
identity if and only if $E$ satisfies condition $D_k$ for some
$k$.

Let $S$ be a commutative inverse semigroup with the set of
idempotents $E$. Consider $\ell^1(S)$ as an $\ell^1(E)$-module with
the following action:

\begin{equation}\label{61} \delta_e\cdot\delta_s
=\delta_s\cdot\delta_e = \delta_s * \delta_e=
\delta_{se},\quad(s\in S, e\in E).
\end{equation}

\begin{theorem}\label{tn}
Let $n\in\mathbb{N}$ and let $S$ be a commutative inverse
semigroup with the set of idempotents $E$. Then $\ell^1(S)$ is
$n$-weakly module amenable as an $\ell^1(E)$-module with the
actions  \emph{(6.1)}.
\end{theorem}
\begin{proof}For any semigroup $S$, group algebra $\ell^{1}(S)$ is commutative if and only if $S$ is
commutative. Since $\ell^1(S)$ is a bi-commutative Banach
$\ell^1(S)$-$\ell^1(E)$-module, so is $\ell^{1}(S)^{(n)}$. By
\cite [Theorem 3.1]{am3}, $\ell^1(S)$ is weakly module amenable
as an $\ell^1(E)$-module. The semigroup algebra $ \ell ^{1}(S)$ is
essential, in fact $ \ell ^{1}(S)= \ell ^{1}(S)\star\ell
^{1}(E)\subseteq\ell ^{1}(S)\star\ell ^{1}(S)\subseteq\ell
^{1}(S)$ (see the proof of \cite[Theorem 3.15]{bab}). Now, it follows from Lemma \ref{deff} that every module derivation from
$\ell^{1}(S)$ into $\ell^{1}(S)^{(n)}$ is zero. This shows that
$\ell^1(S)$ is $n$-weakly module amenable.
\end{proof}

In \cite{po}, the authors proved that the Banach algebra
$\mathcal T_0=\left[\begin{array}{cc}
                                           \ell ^{1}(S) & \ell ^{1}(S) \\
                                            &  \ell ^{1}(S)\\
                                         \end{array}
                                       \right]$ is weak $\mathfrak T_0$-module amenable in which $\mathfrak T_0:=\Big\{\left[
                                         \begin{array}{cc}
                                          \alpha &  \\
                                            &  \alpha\\
                                         \end{array}
                                       \right]~|~\alpha\in\ell
^{1}(E)\Big\}$, where $S$ is a unital commutative inverse
semigroup. The condition of being unital for $S$ is rather strong and it
can be replaced by the weaker condition that $E$ satisfies condition $D_k$ for some $k$.
\begin{example}\label{123} Let $S$ be a commutative inverse semigroup such that $E$ satisfies condition
$D_k$ for some $k$. By
Theorem \ref{t2} we have
\begin{equation*}\label{}
    \mathcal{H}_{\mathfrak T_0}^1(\mathcal T_0, \mathcal T_0^{(2n-1)})\simeq\mathcal{H}_{\ell ^{1}(E)}^1(\ell ^{1}(S),\ell ^{1}(S)^{(2n-1)})\oplus\mathcal{H}_{\ell ^{1}(E)}^1(\ell ^{1}(S), \ell ^{1}(S)^{(2n-1)}).
\end{equation*}
where $\mathcal T_0$ is a Banach $\mathfrak T_0$-module. Since
$\ell ^{1}(S)$ is $n$-weakly module amenable (Theorem \ref{tn}),
$\mathcal T_0$ is $(2n+1)$-weakly module amenable (as $\mathfrak
T_0$-module) again by Theorem \ref{t2}. On the other hand, $\ell
^{1}(S)$ possess a bounded approximate identity, and thus
$\mathcal{T}_0$ is $(2n)$-weakly module amenable by Corollary
\ref{corf}. Therefore $\mathcal{T}_0$ is permanently weakly
module amenable.
\end{example}

Here, for technical reasons, we let $ \ell ^{1}(E)$ act on $ \ell
^{1}(S)$ by multiplication from right and trivially from left,
that is
$$\delta_e\cdot\delta_s = \delta_s, \,\,\delta_s\cdot\delta_e =
\delta_{se} = \delta_s * \delta_e \hspace{0.3cm}(s \in S,  e \in
E).$$

In this case, the ideal $J$ (see section 2) is the closed linear
span of $\{\delta_{set}-\delta_{st} \quad s,t \in S,  e \in E
\}.$ We consider an equivalence relation on $S$ as follows:
$$s\approx t \Longleftrightarrow \delta_s-\delta_t \in J \hspace{0.2cm} (s,t \in
S).$$

For an inverse semigroup $S$, the quotient ${S}/{\approx}$ is a
discrete group (see \cite{am2} and \cite{pou}). Indeed,
${S}/{\approx}$ is homomorphic to the maximal group homomorphic
image $G_S$ \cite{mn} of $S$ \cite{pou2}. In particular, $S$ is
amenable if and only if $G_S$ is amenable \cite{dun, mn}.
\begin{example} \label{}
The Banach algebra $\mathcal T_0$ is  module Arens regular (as an
Banach $\mathfrak T_0$-module) if and only if $\ell ^{1}(S)$ is
module Arens regular as an $\ell ^{1}(E)$-module with trivial
left action and canonical right action by Theorem \ref{the2}. Now
it follows from \cite[Theorem 3.3]{re} that $\mathcal T_0$ is
module Arens regular if and only if is $G_S$ is finite.
\end{example}
\section*{Acknowledgement}
The authors sincerely thank the anonymous reviewers for their
careful reading, constructive comments and fruitful suggestions
to improve the quality of the first draft. The authors also would like to thank Dr. Massoud Amini, for his
valuable discussions and useful comments.

\end{document}